\documentclass[12pt]{article}
\usepackage{amsmath,amsthm,amscd,amssymb}
\usepackage[mathscr]{eucal}
\usepackage{amssymb}
\usepackage{latexsym}

\theoremstyle{definition}
\newtheorem{Def}{Definition}[section]
\newtheorem{Thm}[Def]{Theorem}
\newtheorem{Prop}[Def]{Proposition}
\newtheorem{Rem}[Def]{Remark}
\newtheorem{Ex}[Def]{Example}
\newtheorem{Cor}[Def]{Corollary}

\newtheorem{Lem}[Def]{Lemma}
\newtheorem{Conj}[Def]{Conjecture}

\DeclareMathOperator{\ord}{ord}

\numberwithin{equation}{section}
\newcommand{\Q}{\mathbb{Q}}
\newcommand{\R}{\mathbb{R}}
\newcommand{\C}{\mathbb{C}}
\newcommand{\Z}{\mathbb{Z}}
\newcommand{\F}{\mathbb{F}}

\newcommand{\hh}{\mathbb{H}}
\newcommand{\spl}{\mathrm{Sp}}

\usepackage{longtable}
\newcommand{\FF}{\mathbb{F}}
\newcommand{\tE}{\widetilde{E}^{(2)}}
\newcommand{\tEE}{\widetilde{E}}
\newcommand{\tX}{\widetilde{X}}
\newcommand{\tM}{\widetilde{M}}
\newcommand{\smat}[4]{\begin{smallmatrix} #1 & #2 \\ #3 & #4 \end{smallmatrix}}

\begin{document}
\title{Weights of the mod $p$ kernel of the theta operators}
\author{Siegfried B\"ocherer, Toshiyuki Kikuta and Sho Takemori}
\maketitle

\noindent
{\bf 2010 Mathematics subject classification}: Primary 11F33 $\cdot$ Secondary 11F46\\
\noindent
{\bf Key words}: Siegel modular forms, Congruences for modular forms, Fourier coefficients, Ramanujan's operator, Filtration.

\begin{abstract}
We give some relations between the weights and the prime $p$ of elements of the mod $p$ kernel of the generalized theta operator $\Theta ^{[j]}$. In order to construct examples of the mod $p$ kernel of $\Theta ^{[j]}$ from any modular form, we introduce new operators $A^{(j)}(M)$ and show the modularity of $F|A^{(j)}(M)$ when $F$ is a modular form. Finally, we give some examples of the mod $p$ kernel of $\Theta ^{[j]}$ and the filtrations of some of them.
\end{abstract}
\section{Introduction}
Serre \cite{Se} developed the theory of $p$-adic and congruences for modular forms in one variable.
In his paper, he showed some interesting properties of the $\theta$-operator (Ramanujan's operator) acting on the $q$-expansions $\widetilde{f}=\sum _{n\ge 0}\widetilde{a_f(n)}q^n\in \mathbb{F}_{p}[\![q]\!]$ of modular forms mod $p$ defined as
\[\theta =q\frac{d}{dq} :\widetilde{f}\mapsto \theta (\widetilde{f}) =\sum _{n\ge 1}\widetilde{na_f(n)}q^n. \]
As the one of his results, he showed that the weights (more precisely, the  filtrations) of all elements of the mod $p$ kernel of $\theta $ are divisible by $p$ for the case of level $1$. Moreover, Katz \cite{Kat} showed that this property holds for the case of general level.

The first author and Nagaoka \cite{Bo-Na} extended the notion of the $\theta$-operator to the case of Siegel modular forms of degree $n$. For this operator $\Theta $ (defined in \cite{Bo-Na}),
several people found examples in the mod $p$ kernel of $\Theta $-operator (i.e., Siegel modular forms $F$ satisfying $\Theta (F)\equiv 0$ mod $p$). The first author \cite{Bo} observed that the Klingen-Eisenstein series of weight $12$ arising from Ramanujan's $\Delta$ function is such an example for $p=23$. After this, Mizumoto \cite{Mizu} found another example of weight $16$ and $p=31$, which comes from the Klingen-Eisenstein series arising from a cusp form of weight $16$. Recently, the authors, Kodama and Nagaoka \cite{Bo-Ko-Na,Ki-Ko-Na,Na,Na-Ta} constructed families of such examples of weight $\frac{n+p-1}{2}$ (resp. $\frac{n+3p-1}{2}$) and degree $n$ if the weight is even (resp. odd).

A new feature in the Siegel case is that one should study also vector valued
generalizations $\Theta^{[j]}$ of theta
operators for $0\leq j\leq n$; their $p$-adic properties were given in \cite{Bo-Na2}, e.g. $\Theta^{[1]}$
maps a Siegel modular form $\sum_T a_F(T)q^T$ to a formal series $\sum_T Ta_F(T)q^T$
with coefficients in symmetric matrices of size $n$.

In this paper, we discuss the necessity (as in the one variable cases) of the
relation between the weight and the prime $p$ for an element of the mod $p$
kernel of the generalized theta operators $\Theta ^{[j]}$, in the case where
the weight is small compared with $p$. We remark that Yamauchi \cite{Ya} and
Weissauer \cite{We} also studied the necessity in the special cases
$\Theta^{[1]}$ or $\Theta $. Moreover we construct elements of the mod $p$ kernel of
$\Theta ^{[j]}$ from arbitrary modular form. In order to do this, we
introduce an operator $A^{(j)}(M)$ and study its properties (see Section
\ref{Sec:AOP}). Finally, we give some examples of the mod $p$ kernel of
$\Theta ^{[j]}$ and introduce the filtrations of some of them (Section \ref{Ex:0}, \ref{Table}).
%, \ref{sec:example-filtr-case}).

%%%%%%%%%%%%%%%%%%%%%%%%%%%%
\section{Preliminaries}
\subsection{Siegel modular forms}
\label{sec:siegel-modular-forms}
We denote by $\hh_{n}$ the Siegel upper half space of degree $n$.
We define the action of the symplectic group $\spl_{n}(\R)$ on $\hh_{n}$ by
$gZ = (AZ + B)(CZ + D)^{-1}$ for $Z \in \hh_{n}$, $g \in \spl_{n}(\R)$.
For a holomorphic function $F:\mathbb{H}_n\longrightarrow \mathbb{C}$ and a matrix $g=\left( \begin{smallmatrix} A & B \\ C & D \end{smallmatrix}\right)\in \spl_{n}(\R)$, we define the slash operator in the usual way;
\[F|_k\; g=j(g, Z)^{-k}F(gZ),\]
where $j(g, Z)$ is defined by $\det (CZ+D)$.

Let $N$ be a natural number. In this paper, we deal with three types congruence subgroups of Siegel modular group $\Gamma _n={\rm Sp}_n(\mathbb{Z})$ as follows:
\begin{align*}
&\Gamma ^{(n)}(N):=\left\{ \begin{pmatrix}A & B \\ C & D \end{pmatrix}\in \Gamma _n \: \Big{|} \: B\equiv C \equiv 0_n \bmod{N},\ A\equiv D\equiv 1_n \bmod{N} \right\},\\
&\Gamma _1^{(n)}(N):=\left\{ \begin{pmatrix}A & B \\ C & D \end{pmatrix}
\in \Gamma _n \: \Big| \: C\equiv 0_n \bmod{N},\ A \equiv D \equiv 1_n \bmod{N} \right\},\\
%&\Gamma _{1,\det }^{(n)}(N):=\left\{ \begin{pmatrix}A & B \\ C & D \end{pmatrix}
%\in \Gamma _n \: \Big| \: C\equiv 0_n \bmod{N},\ \det A \equiv \det D \equiv 1 \bmod{N} \right\},\\
&\Gamma _0^{(n)}(N):=\left\{ \begin{pmatrix}A & B \\ C & D \end{pmatrix}\in \Gamma _n \: \Big| \: C\equiv 0_n \bmod{N} \right\}.
\end{align*}

Let $\Gamma $ be the one of above modular groups of degree $n$ with level $N$. For a natural number $k$ and a Dirichlet character $\chi : (\mathbb{Z}/N\mathbb{Z})^\times \rightarrow \mathbb{C}^\times $, the space $M_k(\Gamma ,\chi )$ of Siegel modular forms of weight $k$ with character $\chi$ consists of all of holomorphic functions $F:\mathbb{H}_n\rightarrow \mathbb{C}$ satisfying
\begin{equation}
F|_{k}\: g=\chi (\det D)F(Z)\quad \text{for}\quad g=\begin{pmatrix}A & B \\ C & D \end{pmatrix}\in \Gamma.
\label{modular}\end{equation}
If $n=1$, the usual condition in the cusps should be added.

If $k=l/2$ is half-integral, then we assume that the level $N$ of $\Gamma $ satisfies $4\mid N$.
For $g \in \Gamma_{0}^{(n)}(4)$, we put
$$j_{1/2}(g, Z):=\theta^{(n)}(gZ)/\theta^{(n)}(Z),$$
where
\begin{equation*}
  \theta ^{(n)}(Z):=\sum _{X\in \mathbb{Z}^n}e^{2\pi i {}^t\!X Z X}.
\end{equation*}
Then it is known that
\[j_{1/2}(g,Z)^2=\left( \frac{-4}{\det D} \right) \det (CZ+D) \quad \text{for} \quad g=\begin{pmatrix}A & B \\ C & D\end{pmatrix} \in \Gamma _0^{(n)}(4),  \]
where $\left(\frac{-4}{*}\right)$ is the Kronecker character for the discriminant $-4$.

We define the slash operator for a holomorphic function $F:\mathbb{H}_n\rightarrow \mathbb{C}$ by
\begin{align*}
F|_k \: g:=j_{1/2}(g,Z)^{-l} F(gZ)\quad \text{for}\quad g=\begin{pmatrix}A & B \\ C & D \end{pmatrix}\in \Gamma .
\end{align*}
We define $M_k(\Gamma ,\chi )$ as the space of all of holomorphic functions $F:\mathbb{H}_n\rightarrow \mathbb{C}$ such that
\begin{align*}
F|_k \: g=\chi (\det D) F(Z)\quad \text{for}\quad g=\begin{pmatrix}A & B \\ C & D \end{pmatrix}\in \Gamma ,
\end{align*}
For more details on Siegel modular forms of half-integral weight, see \cite{And-Zu}.

In both cases, when $\chi $ is a trivial character, we write simply $M_k(\Gamma )$ for $M_k(\Gamma ,\chi )$.
Any $F \in M_k(\Gamma, \chi )$ has a Fourier expansion of the form
\[
F(Z)=\sum_{0\leq T\in\frac{1}{N}\Lambda_n}a_F(T)q^T,\quad q:=e^{2\pi i {\rm tr}(TZ)},
\quad Z\in\mathbb{H}_n,
\]
where
\[
\Lambda_n
:=\{ T=(t_{ij})\in {\rm Sym}_n(\mathbb{Q})\;|\; t_{ii},\;2t_{ij}\in\mathbb{Z}\; \}
\]
(the lattice in ${\rm Sym}_n(\mathbb{R})$ of half-integral, symmetric matrices). In particular, if $\Gamma $ satisfies $\Gamma \supset \Gamma _1^{(n)}(N)$, then the
Fourier expansion of $F$ is given by the form
\[
F(Z)=\sum_{0\leq T\in\Lambda_n}a_F(T)q^T.
\]

We denote by $\Lambda _n^+$ the set of all positive definite elements
of $\Lambda _n$. For a subring $R$ of $\mathbb{C}$, let $M_{k}(\Gamma ,\chi )_{R}
\subset M_{k}(\Gamma, \chi )$
denote the $R$-module of all modular forms whose Fourier coefficients are in $R$.

%%%%%%%%%%%%%%%%%%%%%%%%%%%%%%%%%%%%%%%%%%%%%%
\subsection{Vector valued Siegel modular forms}
For later use, we introduce the notion of vector valued Siegel modular forms briefly.
Let $\Gamma \subset \Gamma_{n}$ be the one of subgroups of level $N$ introduced in
Subsection \ref{sec:siegel-modular-forms} and
$\rho : \mathrm{GL}_{n}(\C) \rightarrow \mathrm{GL}_{\C}(V_{\rho})$
a polynomial representation of $\mathrm{GL}_{n}(\C)$.
A holomorphic function $F : \mathbb{H}_{n} \rightarrow V_{\rho}$ is said to be a vector valued Siegel
modular form of automorphy factor $\rho$ and of level $\Gamma$ if and only if $F$ satisfies the
following property:
\begin{equation*}
  F(g Z) = \rho\left(CZ + D\right) F(Z) \quad \text{for all }
  g=\begin{pmatrix}A & B \\ C & D \end{pmatrix}\in \Gamma.
\end{equation*}
If $n = 1$, we add the cusp condition.
%We denote by $M_{\rho}(\Gamma)$ by the space of vector valued Siegel modular forms of weight $\rho$ and of level $\Gamma$.

As in the scalar valued case, a vector valued Siegel modular form $F$ has the following
Fourier expansion:
\begin{equation*}
  F(Z)=\sum_{0\leq T\in\frac{1}{N}\Lambda_n}a_F(T)q^T,
  \quad Z\in\mathbb{H}_n, \ a_{F}(T)\in V_{\rho}.
\end{equation*}

%%%%%%%%%%%%%%%%%%%%%%%%%%%%%%%%%%%%%%%%%
\subsection{Congruences for modular forms}
For a prime $p$ we denote by $\nu_p$ the usual additive $p$-valuation of ${\mathbb Q}$,
normalized by $\nu_p(p)=1$. For a formal power series $F=\sum_{T\in \frac{1}{N}\,\Lambda_n} a_F(T)q^T$ with coefficients in ${\mathbb Q}$ we define
$$\nu_p(F):=\inf \left \{\nu_p(a_F(T)) \;\Big{|}\; T\in \frac{1}{N}\, \Lambda_n\right\}.$$
If the power series is actually the Fourier expansion of a nonzero modular form, then $\nu_p(F)$
is always finite.
%Furthermore, we may tacitly extend the definition above to the case of arbitrary modular forms by extending the evaluation $\nu_p$ to the field generated by all Fourier coefficients.

Let $F_1$, $F_2$ be two formal power series of the forms $F_i=\sum_{0\le T\in \frac{1}{N}\Lambda _n}a_{F_i}(T)q^T$ with $a_{F_i}(T)\in {\mathbb Q}$.  We write
  \begin{equation*}
   F_1 \equiv F_2 \mod{p},
  \end{equation*}
  if and only if
$$\nu_p(F_1-F_2)>\nu_p(F_1).$$
If $\nu_p(F_1)=0$, this means
$a_{F_1}(T) \equiv a_{F_2}(T)$ mod $p$ for all $T \in \frac{1}{N}\Lambda_{n}$
  with $T \ge 0$.

Let $p$ be a prime and $\widetilde{M}_k(\Gamma )_{p^l}$ the space of modular forms mod $p^l$ for $\Gamma $ defined as
\[\widetilde{M}_k(\Gamma )_{p^l}:=\{\widetilde{F} \;|\; F\in M_k(\Gamma )_{\mathbb{Z}_{(p)}}\},\]
where $\widetilde{F}:=\sum _{T}\widetilde{a_F(T)}q^T$ and $\widetilde{a_F(T)}:=a_F(T)$ mod $p^l$.
We put
\[\widetilde{M}(\Gamma )_{p^l}:=\sum _{k\in \mathbb{Z}_{\ge 0}}\widetilde{M}_k(\Gamma)_{p^l}. \]

Let $\omega _N$ be the filtration of modular forms mod $p$ introduced by Serre, Swinnerton-Dyer: For a formal power series of the form $F=\sum _{T\in \Lambda _n}a_F(T)q^T$ (not constant modulo $p$) with $a_F(T)\in \mathbb{Z}_{(p)}$, we define
\[\omega _N(F):=\inf \{k \in \mathbb{Z}_{\ge 1} \;|\; \widetilde{F}\in \widetilde{M}_{k}(\Gamma ^{(n)}_1(N))_p \}. \]
If $F\equiv c$ mod $p$ for some $c\in \mathbb{Z}_{(p)}$, then we regard as $\omega _N(F)=0$.

It follows from \cite{Bo-Na3,Ich} immediately that
\begin{Prop}
\label{omega}
Let $p$ be an odd prime, $N$ a positive integer with $p\nmid N$ and $F\in M_k(\Gamma _1^{(n)}(N))_{\mathbb{Z}_{(p)}}$. Then
\[\omega _N(F)\equiv k \bmod{p-1}. \]
In particular, if $k<p$ and $F\not \equiv c$ mod $p$ for any $c\in \mathbb{Z}_{(p)}$, then $k=\omega _N(F)$.
\end{Prop}
\begin{Def}
A formal power series of the form $F=\sum _{T\in \Lambda _n}a_F(T)q^T$ with $a_F(T)\in \mathbb{Z}_{p}$ is called a \textit{ $p$-adic modular form (of degree $n$)} if there exists a sequence $\{ G_{l}\in M_{k_l}(\Gamma _n)_{\mathbb{Z}_{(p)}} \}$ of modular forms such that
\begin{align*}
\lim _{l\to \infty }G_{l}=F \ \ (p\text{-adically}),
\end{align*}
in other words,
\begin{align*}
\nu_p(F-G_l)\to \infty \ \ (l\to \infty ).
\end{align*}
\end{Def}
\begin{Thm}[\cite{Bo-Na2}]
\label{p-adic}
 Let $p$ be a prime with $p\ge n+3$.
Then any $F\in M_k(\Gamma _0^{(n)}(p^m))_{\mathbb{Z}_{(p)}}$ ($m\ge 0$) is a $p$-adic modular form. In particular, we have $\widetilde{M}(\Gamma_0^{(n)}(p^m))_{p^l}\subset \widetilde{M}(\Gamma _n)_{p^l}$ for any $l\ge 0$ and $m\ge 1$.
\end{Thm}
%%%%%%%%%%%%%%%%%%%%%%%%
\subsection{Theta operators and their properties}

To define the operators $\Theta^{[j]}$ we need some notation: For a symmetric matrix $T$ of size $n$ we
denote by $T^{[j]}$ the matrix of size $\binom{n}{j}  \times \binom{n}{j}$ whose entries are given by the determinants of all submatrices of size $j$.
Then we can explain $\Theta^{[j]}$ by
$$F=\sum_T a_F(T) q^T\longmapsto \Theta^{[j]}(F):=\sum_T T^{[j]}\cdot a_F(T) q^T$$
for any formal power series of the type above.

These operators were introduced  in \cite{Bo-Na2},
and it was shown, that they define (vector valued) modular forms mod $p$, when applied
to modular forms.  Note that $\Theta ^{[n]}$ is the $\Theta $-operator defined in \cite{Bo-Na}, i.e., for a formal Fourier series $F=\sum _{T\in \Lambda _n}a_F(T)q^T$, we have
\begin{align*}
\Theta^{[n]}(F)=\sum_{T\in \Lambda _n}(\det T )a_F(T)q^T.
\end{align*}
Hence, we write simply as $\Theta $ for $\Theta ^{[n]}$.

For $0\le j\le n$, we observe the following (obvious) properties of these operators with respect to congruences:

\begin{Prop}
\label{ThetaOp}
(1) $\Theta ^{[j]}(F)\not \equiv 0$ mod $p$ is equivalent to the existence of $T\in \Lambda _n$ and a $j\times j$ submatrix $R$ of $T$ such that $a_F(T)\not \equiv 0$ mod $p$ and  $p\nmid \det R$.

(2) $\Theta ^{[j]}(F)\equiv 0$ mod $p$ implies $\Theta ^{[j+1]}(F)\equiv 0$ mod $p$.

(3) $\Theta ^{[j]}(F)$ is a mod $p$ singular if and only if $T^{[j]}a_F(T)\equiv 0$ mod $p$ for all $T\in \Lambda _n^+$.

(4) Let $j\ge 0$ be an integer. If $\Theta ^{[j]}(F)$ is mod $p$ singular with $p$-rank $r_p$, then  $\Theta ^{[r_p+1]}(F)\equiv 0$ mod $p$.
\end{Prop}

The following theorem is due to Katz:
\begin{Thm}[Katz \cite{Kat} (cf. Serre \cite{Se})]
\label{Kat}
Let $p$ be an arbitrary prime and $N$ a positive integer such that $N\ge 3$ and $p\nmid N$. For $f\in M_k(\Gamma _1^{(1)}(N))_{\mathbb{Z}_{(p)}}$ ($k\in \mathbb{Z}_{\ge 1}$), suppose that $\Theta ^{[1]}(f)\equiv 0$ mod $p$. Then we have $p\mid \omega_N (f)$.
\end{Thm}
\begin{Rem}
In this case (the degree is $1$), the operator $\Theta ^{[1]}$ is the usual Ramanujan's operator $\theta $.
\end{Rem}
\begin{Thm}[\cite{Bo-Na}]
\label{TOP2}
Let $p$ be a prime with $p\ge n+3$. If $F\in M_k(\Gamma _n)_{\mathbb{Z}_{(p)}}$, then $\widetilde{\Theta (F)}\in \widetilde{M}_{k+p+1}(\Gamma _n)_p$. In particular $\omega _1(\Theta (F))\le k+p+1$ holds. Therefore we can regard as
\[\Theta :\widetilde{M}(\Gamma _n)_p\longrightarrow \widetilde{M}(\Gamma _n)_p. \]
\end{Thm}
We introduce a relation between mod $p$ singular modular forms and the mod $p$ kernel of $\Theta ^{[j]}$:
\begin{Prop}
\label{kernel3}
Let $p$ be an odd prime and $N$ a positive integer with $p\nmid N$.
For a positive integer $k$, assume that $F\in M_{k}(\Gamma ^{(n)}(N))_{\mathbb{Z}_{(p)}}$ is mod $p$ singular of $p$-rank $r_p$ with $k\not \equiv \frac{r_p}{2}$ mod $2$. Then we have $\Theta ^{[r_p]}(\Phi ^{(n-r_p)}(F))\equiv 0$ mod $p$. Here $\Phi $ is the Siegel $\Phi $-operator and hence $\Phi ^{(n-r_p)}(F)\in M_{k}(\Gamma ^{(r_p)}(N))$.
\end{Prop}
\begin{Rem}
If $F\in M_k(\Gamma ^{(n)}(N))$ is mod $p$ singular of $p$-rank $r_p$, then we have $2k-r_p\equiv 0$ mod $p-1$ by the result of \cite{Bo-Ki}.
Therefore in this case $r_p$ should automatically be even.
\end{Rem}
\begin{proof}
By taking $\Phi $-operator several times, we may suppose that the $p$-rank $r_p$ is $n-1$.
Then it suffices to prove that $\Theta ^{[n-1]}(\Phi (F))\equiv 0$ mod $p$.
Moreover, by taking $F(NZ)\in M_k(\Gamma _1^{(n)}(N^2))$ when $F$ is of $\Gamma ^{(n)}(N)$, it suffices to prove it for $F\in M_{k}(\Gamma _1^{(n)}(N))$.

For any $T_0\in \Lambda _{n}$ with ${\rm rank }(T_0)=n-1$ satisfying $a_F(T_0)\not \equiv 0$ mod $p$, we may assume that $T_0=\left( \begin{smallmatrix}0 & 0 \\ 0 & M_0\end{smallmatrix} \right)$ for some $M_0\in \Lambda _{n-1}^+$.
We shall prove $p\mid \det M_0$ for any such $M_0$.

Recall that $F$ has a Fourier-Jacobi expansion of the form
\[F(Z)=\sum _{M\in \Lambda _{n-1}}\varphi _T(\tau , z)e^{2 \pi i {\rm tr}(M\cdot \tau ')}.\]
Here, we decomposed $\mathbb{H}_{n}\ni Z=\left( \begin{smallmatrix} \tau &  {}^t \frak{z} \\ \frak{z} & \tau ' \end{smallmatrix} \right) $ for $\tau \in \mathbb{H}_1$, $\tau ' \in \mathbb{H}_{n-1}$.
Pick up $M_0$-th Fourier-Jacobi coefficient and consider its theta expansion;
\[ \varphi _{M_0}(\tau ,\frak{z})=\sum _{ \mu } h_\mu (\tau )\Theta _{M_0}[\mu ](\tau,\frak{z}) \]
where $\mu $ runs over all elements of $\mathbb{Z}^{(1,n-1)} \cdot (2M_0) \backslash \mathbb{Z}^{(1,n-1)}$ and
\begin{align*}
h_\mu (\tau )=\sum _{l=0}^\infty a_F\left( \begin{smallmatrix} l & \frac{\mu}{2} \\ \frac{{}^t \mu}{2} & M_0 \end{smallmatrix}\right)e^{2\pi i (l-\frac{1}{4}M_0^{-1}[{}^t \mu ])\tau }.
\end{align*}
From the mod $p$ singularity of $F$, the above $h_{0}$ satisfies that
\[h_0 \equiv c\not \equiv 0 \bmod{p}\]
Moreover $h_{0}$ is a modular form of weight $k-\frac{n-1}{2}\in \mathbb{Z}_{\ge 1}$
for $\Gamma _0^{(1)}(NL)$ by Lemma 5.1 in \cite{Bo-Ki}.
Here $L$ is the level of $M_0$.
If $p\nmid L$ then we have $k-\frac{n-1}{2}\equiv 0$ mod $p-1$ by Katz \cite{Kat}.
However this is impossible because of $k\not \equiv \frac{r_p}{2}$ mod $2$.
Hence $p\mid L$ follows. In particular we have $p\mid \det M_0$.
This completes the proof of Proposition \ref{kernel3}.
\end{proof}

%%%%%%%%%%%%%%%%%%%%%%%
\section{Main results and their proofs}
\subsection{Main results}
\label{Main}
For any $T\in \Lambda _n$, we denote by $\varepsilon (T)$ the content of $T$ defined as
\[\varepsilon (T):={\rm max} \{d \in \mathbb{Z}_{\ge 1} \; | \; d^{-1}T \in \Lambda _n\}. \]
Let $F$ be a scalar valued modular form. If $F\not \equiv 0$ mod $p$ and $\Theta (F)\equiv 0$ mod $p$, then there are three possibilities as follows;

(a) For any $T\in \Lambda _n^+$ we have $a_F(T)\equiv 0$ mod $p$.

(b) For any $T\in \Lambda _n^{+}$ with $a_F(T)\not \equiv 0$ mod $p$, we have $p\mid \varepsilon(T)$.

(c) There exists $T\in \Lambda _n^+$ such that $a_F(T)\not \equiv 0$ mod $p$ and $p\nmid \varepsilon (T)$. \\
A modular form $F$ of the type (a) is called ``mod $p$ singular".
In this case, the authors discussed the possible weight in \cite{Bo-Ki}.
Therefore, the main purpose of this paper is to consider the types (b) and (c).

The first main result concerns $F$ of the type (b), but under a condition on $k$.
Note that the condition (b) is equivalent to $\Theta ^{[1]}(F)$ is mod $p$ singular, i.e. the vector valued modular form $\Theta ^{[1]}(F)$  (mod $p$) satisfies the condition (a) above, because of Proposition \ref{ThetaOp} (3).
\begin{Thm}
\label{Main2}
Let $p$ be a prime with $p\ge 3$ and $N$ a positive integer satisfying that $N\ge 3$ and $p\nmid N$ or $N=1$.
For a positive integer $k$, let $F\in M_k(\Gamma _1^{(n)}(N))_{\mathbb{Z}_{(p)}}$.
Assume that $F\not \equiv c$ mod $p$ for any $c\in \mathbb{Z}_{(p)}$.

(1) If $\Theta ^{[1]}(F)\equiv 0$ mod $p$ and
\begin{align*}
\begin{cases} 0<k<2p-1\ &(k\ \text{odd}),\\
 0<k<3p-1\  &(k\ \text{even})
\end{cases}\
\text{then}\
k=
\begin{cases}
p\ &(k\ \text{odd}), \\
2p\ &(k\ \text{even}).
\end{cases}
\end{align*}

(2) If $\Theta ^{[1]}(F)\equiv 0$ mod $p$, $0<k<p^2-p+1$ and $p\mid k$ then $k=\omega _N(F)$.

(3) If $\Theta ^{[1]}(F)$ is non-trivial mod $p$ singular of $p$-rank $r_p$, then $2k-r_p\equiv 0$ mod $p-1$.
\end{Thm}

\begin{Rem}
(1) A modular form $F$ satisfying $\Theta ^{[1]}(F)\equiv 0$ mod $p$ is called ``totally $p$-singular'' by Weissauer \cite{We}.
Weissauer \cite{We} also obtained the similar statements (at least for the case of level $1$)
under a certain condition on the largeness of $p$, in geometrical terminology.
Our statement is phrased in classical (elementary) language. \\
(2) There exists a mod $p$ singular modular form $F$ ($\not \equiv 0$ mod $p$) such that $\Theta ^{[1]}(F)\equiv 0$ mod $p$.
In fact, we can construct such an example in the following way:
For any mod $p$ singular modular form $F\in M_k(\Gamma _n)_{\mathbb{Z}_{(p)}}$, we consider
\[G:=\sum _{T\in \Lambda _n}a_F(pT)q^{pT}\in M_k(\Gamma_0^{(n)}(p^2)).\]
Applying Theorem \ref{p-adic}, we can take $H\in M_{k'}(\Gamma _n)_{\mathbb{Z}_{(p)}}$ such that $H\equiv G$ mod $p$.
Then $H$ is a mod $p$ singular modular form such that $\Theta ^{[1]}(H)\equiv 0$ mod $p$.
For the existence of mod $p$ singular modular forms and for their possible weights, see \cite{Bo-Ki}.
\end{Rem}
The statement (3) in this theorem follows immediately from a property on mod $p$ singular vector valued Siegel modular forms,
which is a generalization of the result in \cite{Bo-Ki}:
\begin{Thm}
\label{vec-val}
Let $p$ be an odd prime and $k$ a positive integer.  Let $N$ be a positive integer with $p\nmid N$ and $F\in M_k(\Gamma_1^{(n)}(N))$.
Suppose that $\Theta^{[j]}(F)$ is mod $p$ singular
of $p$-rank $r_p$.
Then
$$2k\equiv r_p \bmod p-1$$
holds.
\end{Thm}
\begin{Rem}
We may allow the modular group to be of type
$\Gamma_1^{(n)}(N)\cap \Gamma_0^{(n)}(p^l)$ with $N$ coprime to $p$.
We may also allow quadratic nebentypus modulo $p$.
\end{Rem}
The second main result concerns $F$ of the type (c).
We remark that $p\nmid \varepsilon (T)$ for $T\in \Lambda _n$ is equivalent to the existence of $j$ with $1\le j\le n-1$ such that $p\nmid T^{[j]}$,
where we write $p\mid T^{[j]}$ if $p$ divides all entries of $T^{[j]}$, otherwise we write $p\nmid T^{[j]}$.
Moreover the existence of $j$ and $T\in \Lambda _n^+$ with $p\nmid T^{[j]}$ such that $a_F(T)\not \equiv 0$ mod $p$, implies $\Theta ^{[j]}(F)\not \equiv 0$ mod $p$.
Namely we have
\begin{align*}
&\quad \exists T\in \Lambda _n^+\ \text{s.t.}\ a_F(T)\not \equiv 0 \bmod{p}, \ p\nmid \varepsilon(T) \\
\Longleftrightarrow &\quad \exists j\ (1\le j\le n-1),\ \exists T\in \Lambda _n^+\ \text{s.t.}\ a_F(T)\not \equiv 0 \bmod{p}, \ p\nmid T^{[j]}\\
\Longrightarrow &\quad \exists j\ (1\le j\le n-1)\ \text{s.t.}\ \Theta ^{[j]}(F)\not \equiv 0 \bmod{p}.
\end{align*}
Note also that the converse of the last right arrow is not assured in general.

For any $F$ of the type (c), we can find $j$ such that $\Theta ^{[j]}(F)\not \equiv 0$ mod $p$ and $\Theta ^{[j+1]}(F) \equiv 0$ mod $p$.
Then we have the following statement:
\begin{Thm}
\label{Main1}
Let $p$ be a prime with $p\ge 3$ and $N$ a positive integer with $p\nmid N$. Let $n$, $j$ and $k$ be positive integers such that $j<n$. Assume that $F\in M_k(\Gamma _{1}^{(n)}(N))_{\mathbb{Z}_{(p)}}$ satisfies $\Theta ^{[j+1]}(F)\equiv 0$ mod $p$ and $\Theta ^{[j]}(F)\not \equiv 0$ mod $p$.

(1) If
\begin{align*}
&\begin{cases}
k<p+(j-1)/2\ &(j\ \text{odd}),\\
k<2p+(j-2)/2\ &(j\ \text{even},\ k-j/2\ \text{odd}),\\
k<3p+(j-2)/2\ &(j\ \text{even},\ k-j/2\ \text{even})
\end{cases}\quad \text{then}\\
&\begin{cases}
2k-j=p\ &(j\ \text{odd}),\\
k-j/2=p\ &(j\ \text{even},\ k-j/2\ \text{odd}),\\
k-j/2\equiv 0 \bmod{p-1} \quad \text{or}\quad k-j/2=2p\ &(j\ \text{even},\ k-j/2\ \text{even}).
\end{cases}
\end{align*}

(2) If
\[
\begin{cases}
2k-j<p^2-p+1 \ &(j\ \text{odd}), \\
k-j/2<p^2-p  \ &(j\ \text{even})
\end{cases}
\quad \text{and} \quad p\mid (2k-j)
\quad \text{then} \quad k=\omega _N (F).\]
\end{Thm}
In a more general situation, we predict that
\begin{Conj}
\label{Conj}
Let $p$ be a prime and $n$, $j$ and $k$ be positive integers with $j<n$. Let ``$k$ be sufficiently small compared with $p$''.
Assume that $F\in M_k(\Gamma ^{(n)}_{1}(N))_{\mathbb{Z}_{(p)}}$ ($k\in \mathbb{Z}_{\ge 1}$) satisfies $\Theta ^{[j+1]}(F)\equiv 0$ mod $p$ and $\Theta ^{[j]}(F)\not \equiv 0$ mod $p$. Then we have
\[p\mid (2\omega _N(F)-j)\]
\end{Conj}
Therefore, we can regard Theorem \ref{Main1} as an example which supports this conjecture.
\begin{Rem}
\label{b-c}
(1) If the weight is large compared with $p$, the statement of this conjecture is not true.
We will show this for the case of degree $2$ and level $1$, by numerical examples in Subsections \ref{Ex:2} and
\ref{sec:example-filtr-case}. \\
(2) Yamauchi \cite{Ya} concluded the similar statements as in the two theorems above
for the case of degree $2$, without condition on the smallness of $k$ compared with $p$,
but under a certain geometrical non-vanishing condition. The proof is also algebraic geometrical.
\end{Rem}

We summarize the simplest case $n = 2$ in Theorems \ref{Main2} and \ref{Main1} as follows:
\begin{Cor}
Let $p$ be a prime with $p\ge 3$ and $N$ a positive integer with $p\nmid N$. For a positive integer $k$, assume that $F\in M_k(\Gamma _{1}^{(2)}(N))_{\mathbb{Z}_{(p)}}$ satisfies $\Theta (F)\equiv 0$ mod $p$.

(1) If $k<p$ and there exists $T\in \Lambda _2^+$ with $p\nmid \varepsilon (T)$ satisfying $a_F(T)\not \equiv 0$ mod $p$, then we have $2k-1=p$.

(2) If we have $p\mid \varepsilon (T)$ for any $T\in \Lambda _2^+$ with $a_F(T)\not \equiv 0$ mod $p$, then $\Theta ^{[1]}(F)\equiv 0$ mod $p$. In particular, if \begin{align*}
\begin{cases} 0<k<2p-1\ &(k\ \text{odd}),\\
 0<k<3p-1\  &(k\ \text{even})
\end{cases}\
\text{then}\
k=
\begin{cases}
p\ &(k\ \text{odd}), \\
2p\ &(k\ \text{even}).
\end{cases}
\end{align*}
\end{Cor}

%%%%%%%%%%%%%%%%%%%%%%%%%%%
\subsection{Proof of Theorem \ref{vec-val}}
Since Theorem \ref{Main2} (3) follows from Theorem \ref{vec-val}, we start with proving Theorem \ref{vec-val}.

We observe that the Fourier expansion of $\Theta^{[j]}(F)$
runs only over elements of $\Lambda_n$ with ${\rm rank}(T)\geq j$, therefore
only the case $j\leq r_p$ is of interest for us.
Also it may be convenient to reduce the claim to the case $r_p=n-1$
by
applying the Siegel $\Phi$-operator
several times; for details on the Siegel $\Phi$-operator in the vector-valued
case we refer to \cite{Frei,We2}. We just mention that for $j<n$ we may identify
$$\Phi \left(\sum_{T\in \Lambda_n}
T^{[j]}\cdot a_F(T)q^T\right)$$
with
$$\sum_{S\in \Lambda_{n-1}}
S^{[j]}\cdot a_F \left(\begin{array}{cc} 0 & 0\\0 & S\end{array}\right) q^S. $$

We introduce some useful notation following \cite{Frei}: For a $n$-rowed matrix $M$ and two subsets $P$, $Q$ of $\{1,\ldots ,n\}$
with $t$ elements we denote by $M^{[t]}_{(P,Q)}$ the determinant of the $t$-rowed matrix $M^P_Q$ which we obtain from $M$ by deleting
all rows which do not belong to $P$ and all columns which do not belong to $Q$.

Now start from the Fourier expansion $F=\sum a_F(T)q^T$;
there exists $T_0\in \Lambda_n$ with rank $n-1$ such that
$T_0^{[j]}\cdot a_F(T_0)$ is not congruent zero modulo $p$.
We may assume that $T_0$ is of the form
$$T_0= \left(\begin{array}{cc} 0 & 0\\
0 & M_0\end{array}\right)\qquad (M_0\in \Lambda^+_{n-1}). $$
The property of $T_0$ from above implies that there is at least
one entry of the matrix $T_0^{[j]}$ which is not congruent zero modulo $p$, i.e.
there exist subsets $a^o$, $b^o$ of $\{1,\dots, n\}$ with $\det({T_0}^{a^o}_{b^o})
\not \equiv 0$ mod $p$.
From the special shape of $T_0$ it follows that both $a^o$ and $b^o$
are subsets of $\{2,\dots ,n\}$, i.e. the $(a^o,b^o)$-entry of $T_0^{[j]}$
is a determinant of a submatrix of $M_0$, which we call $d_0$.

We decompose $Z\in {\mathbb H}_n$ as
$Z=\left(\begin{smallmatrix} \tau & {\mathfrak z}\\ {\mathfrak z}^t & \tau'\end{smallmatrix}\right)$ with $\tau'\in {\mathbb H}_{n-1}$ and
study the
Fourier Jacobi coefficient
$$\varphi_{M_0}(\tau,\frak{z})e^{2\pi i {\rm tr}(M_0\tau')}$$
viewed as a subseries of the Fourier expansion of $F$.
We apply the operator $\Theta^{[j]}$ to
this subseries of $F$;
then its $(a^o,b^o)$ entry is just
$$d_0\cdot \varphi_{M_0}(\tau,\frak{z})e^{2\pi i {\rm tr}(M_0\tau')}$$
Now we may proceed as in \cite{Bo-Ki}:
For all $R\in \Lambda_n$  the $R$-th Fourier coefficients of this
series should be congruent zero modulo $p$
unless $R$ is of rank $n-1$.
This implies that in the theta expansion of $\varphi_{M_0}$
a modular form $h_0$ of weight $k-\frac{n-1}{2}$ appears, which should
be congruent to a (nonzero) constant modulo $p$.
The requested congruence follows from this as in \cite{Bo-Ki}.
\qed

%%%%%%%%%%%%%%%%%%%%%%%%%%%%%%%
\subsection{Proof of Theorem \ref{Main2}}
(1) From Proposition \ref{ThetaOp}, the assumption $\Theta ^{[1]}(F)\equiv 0$ mod $p$ implies that,
for any $T\in \Lambda _n$ satisfying $a_F(T)\not \equiv 0$ mod $p$, we have $p\mid \varepsilon(T)$.
In particular, all diagonal components of such $T$ are divisible by $p$.
We fix one of them such that $T\neq 0_n$ and denote it by $T_0$.
The existence of such $T_0$ is guaranteed by the assumption $F\not \equiv c$ mod $p$.
Let ${\rm diag}(T_0)=(d_1,d_2,\cdots ,d_{n})$ (with $p\mid d_i$ for any $i$).
Since $T_0\neq 0_n$, we may assume that $d_1>0$ by changing $T_0$ in its ${\rm GL}_n(\mathbb{Z})$-equivalence class.

Consider the integral extract for $T_0$ defined by the first author and Nagaoka \cite{Bo-Na3};
\begin{align*}
&f(\tau ):=\sum_l a(l)q^l,\\
 &a(l):=\sum_T a_F(T),
\end{align*}
where $T$ runs over all positive semi-definite elements of $\Lambda _n$ satisfying
\[T\equiv T_0 \bmod{M},\quad {\rm diag}(T)=(l,d_2,\cdots , d_n) \]
and $M$ is a large enough such that $(p\det T_0 ,M)=1$ and $a(d_1)\equiv a_f(T_0)\not \equiv 0$ mod $p$.
Note that $f\in M_k(\Gamma _1^{(1)}(NM^2))$.
Then $p\mid l$ when $a(l)\not \equiv 0$ mod $p$ and $f\not \equiv c$ mod $p$ for any $c\in \mathbb{Z}_{(p)}$.
Hence we have $\Theta ^{[1]}(f)\equiv 0$ mod $p$ and $\omega_{NM^2} (f)>0$.
Applying Theorem {\ref{Kat}}, we obtain $p\mid \omega _{NM^2}(f)$.
By Proposition \ref{omega}, $\omega _{NM^2}(f)\equiv k$ mod $p-1$ and
therefore $\omega _{NM^2}(f)$ and $k$ have the same parity.

If $k$ is odd, then $0<\omega _{NM^2}(f)\le k<2p-1$ and $p\mid \omega _{NM^2}(f)$.
In this case we have $\omega _{NM^2}(f)=p$.
Therefore $k=p$ or $k=2p-1$.
By the assumption $k<2p-1$, we obtain $k=p$.

If $k$ is even, then $\omega _{NM^2}(f)$ is even such that
$0<\omega _{NM^2}(f)\le k<3p-1$ and $p\mid \omega _{NM^2}(f)$.
In this case we have $\omega _{NM^2}(f)=2p$.
Therefore $k=2p$ or $k=3p-1$.
Since the assumption $k<3p-1$, we obtain $k=2p$.
This completes the proof of (1) in Theorem \ref{Main2}. \\

\noindent
(2) As in the proof of (1), we can take $f\in M_k(\Gamma _1^{(1)}(NM^2))$ ($p\nmid M$) such that
$\Theta^{[1]}(f)\equiv 0$ mod $p$, $f\not \equiv c$ mod $p$ for any $c\in \mathbb{Z}_{(p)}$ and
\[0<\omega _{NM^2}(f)\le \omega _{N}(F)\le  k. \]
To prove $k=\omega _N(F)$, we may prove $k=\omega _{NM^2}(f)$.
Now we obtain $p\mid \omega _{NM^2}(f)$ by Theorem {\ref{Kat}}.
By the assumption, we have $\omega _{NM^2}(f)\equiv k$ mod $p(p-1)$, because $\omega _{NM^2}(f)\equiv k\equiv 0$ mod $p$ and $\omega _{NM^2}(f)\equiv k$ mod $p-1$.
Then there exists $t\ge 0$ such that $k=\omega _{NM^2}(f)+tp(p-1)$.
However, from $0<k<p^2-p+1$, we have $t=0$ and hence $k=\omega _{NM^2}(f)$. This completes the proof of (2) in Theorem \ref{Main2}.\\

\noindent
(3) The statement follows immediately from Theorem \ref{vec-val}.
\qed

%%%%%%%%%%%%%%%%%%%%%%%%%%%5
\subsection{Proof of Theorem \ref{Main1}}
(1)  By Proposition \ref{ThetaOp} and the assumption $\Theta ^{[j]}(F)\not \equiv 0$ mod $p$,
there exist $T\in \Lambda _n$ and a $j\times j$ submatrix $M$ of $T$ such that $a_F(T)\not \equiv 0$ mod $p$ and $p\nmid  \det M$.

By changing $T$ in its ${\rm GL}_n(\mathbb{Z})$-equivalence class, we may assume that there is a principal submatrix $M$ of size $j$ in $T$
such that $p\nmid \det M$, where principal matrix is defined as a matrix obtained by omitting the same columns and rows.
In particular, $T$ is equivalent over ${\rm GL}_n(\mathbb{Z})$ to a matrix of the form
$$T_0=\begin{pmatrix} 0 & 0 \\ 0 & M_0\end{pmatrix}$$
with $M_0\in \Lambda _{j}$, $p\nmid \det M_0$ and $a_F(T_0)\not \equiv 0$ mod $p$.

The Fourier-Jacobi expansion of $F$ can be written in the form
\[F(Z)=\sum _{M\in \Lambda _{j}}\varphi _M(\tau , \frak{z})e^{2 \pi i {\rm tr}(M\cdot \tau ')}.\]
Here, we decomposed $\mathbb{H}_n\ni Z=\left( \begin{smallmatrix} \tau & {}^t \frak{z} \\ \frak{z} & \tau ' \end{smallmatrix} \right) $ for $\tau \in \mathbb{H}_{n-j}$ and $\tau '\in \mathbb{H}_{j}$.

We consider the $M_0$-th Fourier-Jacobi coefficient
\[ \varphi _{M_0}(\tau ,\frak{z})=\sum _{ \mu } h_\mu (\tau )\Theta _{M_0}[\mu ](\tau,\frak{z}), \]
where $\mu $ runs over all elements of $\mathbb{Z}^{(n-j,j)} \cdot (2M_0) \backslash \mathbb{Z}^{(n-j,j)}$ and
\begin{align*}
h_\mu (\tau )=\sum _{L\in \Lambda _{n-j}} a_F\begin{pmatrix} L & \frac{\mu}{2} \\ \frac{{}^t \mu}{2} & M_0 \end{pmatrix}e^{2\pi i{\rm tr} ((L-\frac{1}{4}M_0^{-1}[{}^t \mu ])\tau )}.
\end{align*}
Then $h_{\mu }$ is a modular form of weight $k-\frac{j}{2}$
for $\Gamma ^{(n-j)}(4NQ)$, where $Q$ is the level of $M_0$ and we have $p\nmid Q$. Hence we have
\[H_{\mu}:=h_{\mu }(4NQ\tau )\in M_{k-\frac{j}{2}}(\Gamma _1^{(n-j)}(4^2N^2Q^2))_{\mathbb{Z}_{(p)}}.\]

Now we prove
\begin{Lem}
\label{lem2}
If the Fourier coefficient of $H_{\mu }$ at $L-\frac{1}{4}M_0^{-1}[{}^t \mu ]$ is nonzero modulo $p$, then all entries of
$$L-\frac{1}{4}M_0^{-1}[{}^t \mu ]$$
are divisible by $p$. In particular, by Proposition \ref{ThetaOp}, we have
\[\Theta ^{[1]}(H_{\mu })\equiv 0 \bmod{p}. \]
\end{Lem}
\begin{proof}
We have by a direct calculation
\begin{align*}
S&=\begin{pmatrix}L & \frac{\mu }{2} \\ \frac{{}^t\mu}{2} & M_0\end{pmatrix}\\
&=\begin{pmatrix}1_{n-j} & \frac{\mu }{2}M_0^{-1} \\ 0 & 1_{j} \end{pmatrix}
\begin{pmatrix}L-\frac{1}{4}M_0^{-1}[{}^t \mu ] & 0 \\ 0 & M_0 \end{pmatrix}
\begin{pmatrix}1_{n-j} & 0 \\ M_0^{-1}\frac{{}^t \mu }{2} & 1_{j} \end{pmatrix}.
\end{align*}
(Multi-) linear algebra shows that for $J:=\{n-j+1,\ldots ,n\}$ and $i$, $i'\in \{1,\ldots , n-j\}$, we have
\[S_{(\{i \}\cup J, \{i' \}\cup J)}^{[j]}=\left(L-\frac{1}{4}M_0^{-1}[{}^t \mu ] \right)^{[1]}_{(i,i')}\cdot \det M_0.\]
Here the notation $S^{[t]}_{(P,Q)}$ is the same as in the proof of Theorem \ref{vec-val}.

We observe that the left hand side is divisible by $p$ and $\det M_0$ is coprime
to $p$, therefore all entries of $L-\frac{1}{4}M_0^{-1}[{}^t \mu ]$ are divisible by $p$.
\end{proof}
We return to the proof of Theorem \ref{Main1}.
Let $j$ be odd. By $\Theta ^{[1]}(H_{\mu })\equiv 0 \bmod{p}$, we can easily prove $\Theta ^{[1]}(H_{\mu }^2)\equiv 0$ mod $p$.
Then $H_{\mu }^2$ is of weight $2k-j<2p-1$ and $2k-j$ is odd.
Therefore $2k-j\not \equiv 0$ mod $p-1$. This implies $H_{\mu }^2\not \equiv c$ mod $p$ for any $c\in \mathbb{Z}_{(p)}$ (see \cite{Bo-Na3}).
Hence we can apply Theorem \ref{Main2} (1) to $H_{\mu }^2$.
We conclude that $2k-j=p$ in the case where $j$ is odd.

Let $j$ be even and $k-j/2$ odd. The weight of $H_{\mu }$ is $k-j/2<2p-1$ and therefore $H_{\mu }\not \equiv c$ mod $p$ for any $c\in \mathbb{Z}_{(p)}$, because of $2k-j\not \equiv 0$ mod $p-1$.  In this case, we can directly apply Theorem \ref{Main2} (1) to $H_{\mu}$.
Hence we obtain $k-j/2=p$.

Let $j$ be even and $k-j/2$ even. Note that $k-j/2<3p-1$ by the assumption.
If $H_{\mu }\equiv c$ mod $p$ for some $c\in \mathbb{Z}_{(p)}$, then we have $k-j/2\equiv 0$ mod $p-1$,
otherwise we can apply Theorem \ref{Main2} (1) to $H_{\mu}$.
Therefore, we obtain $k-j/2\equiv 0$ mod $p-1$ or $k-j/2=2p$.
This completes the proof of (1) in Theorem \ref{Main1}.\\

\noindent
(2) Let $H_{\mu }\in M_{k-\frac{j}{2}}(\Gamma _1^{(n-j)}(4^2N^2Q^2))_{\mathbb{Z}_{(p)}}$ be the function appeared in the proof of (1).
Note that $\omega _{4^2N^2Q^2}(H_{\mu })+j/2\le \omega _N(F)\le k$.
By the assumption, we have
\[
\begin{cases} \omega  _{4^2N^2Q^2}(H_{\mu }^2)\le 2k-j<p^2-p+1 \quad &(j\ \text{odd}),\\
 \omega  _{4^2N^2Q^2}(H_{\mu })\le k-j/2<p^2-p \quad &(j\ \text{even})
\end{cases}
\quad \text{and} \quad p\mid (2k-j).
\]

To apply Theorem \ref{Main2} (2) to $H^2_{\mu}$ ($j$ odd) and $H_{\mu }$ ($j$ even),
we need to confirm that they are not constant modulo $p$.

If $j$ is odd, then $2k-j\not \equiv 0$ mod $p-1$. This implies $H^2_{\mu}\not \equiv c$ for any $c\in \mathbb{Z}_{(p)}$.
Hence we can apply Theorem \ref{Main2} (2) to $H^2_{\mu}$.
It follows that
\[2k-j=\omega _{4^2N^2Q^2}(H_{\mu }^2) \le 2\omega _{4^2N^2Q^2}(H_{\mu })\le 2\omega _N(F)-j\le 2k-j.\]
This indicates $k=\omega _N(F)$.

Let $j$ be even. Assume that $H_{\mu }\equiv c$ mod $p$ for some $c\in \mathbb{Z}_{(p)}$.
Then we have both conditions $k-j/2\equiv 0$ mod $p-1$ and $p\mid (k-j/2)$
because of the assumption of the theorem.
Then there exists $t\ge 1$ such that $2k-j=tp(p-1)$.
However this is impossible because of $2k-j<p^2-p$.
This means that $H_{\mu }\not \equiv c$ mod $p$ for any $c\in \mathbb{Z}_{(p)}$.

Hence we can apply Theorem \ref{Main2} (2) to $H_{\mu}$ and then
\[k=\omega _{4^2N^2Q^2}(H_{\mu })+j/2\le \omega _{4^2N^2Q^2}(F)\le k. \]
 Therefore we obtain $k=\omega _N(F)$.

This completes the proof of (2) in Theorem \ref{Main1}.
\qed

%%%%%%%%%%%%%%%%%%%%%%%%%%%%%%
\section{On operators $A^{(j)}(p)$}
\label{Sec:AOP}
Following Choi-Choie-Richter \cite{C-C-R}, for a Siegel modular form $F\in M_k(\Gamma _n)_{\mathbb{Z}_{(p)}}$ with a Fourier expansion $F=\sum _T a_F(T)q^T$, we define an operator $A(p)$ (their notation is $U(p)$) as
\begin{align*}
F|A(p):=\sum _{\substack{T\in \Lambda _n \\ p|\det T}}a_F(T)q^T.
\end{align*}
We remark that this operator $A(p)$ is different from the usual $U(p)$-type Hecke operator investigated in \cite{Bo2} and elsewhere. If $p\ge n+3$, it is easy to see that $\widetilde{F|A(p)}= \widetilde{FE_{p-1}^{p+1}}-\Theta ^{p-1}(\widetilde{F}) \in \widetilde{M}_{k+p^2-1}(\Gamma _n)_p$, because $\Theta (\widetilde{F})\in \widetilde{M}_{k+p+1}(\Gamma _n)_p$ (see Theorem \ref{TOP2}). Here $E_{p-1}\in M_{p-1}(\Gamma _n)_{\mathbb{Z}_{(p)}}$ is such that $E_{p-1}\equiv 1$ mod $p$ obtained in \cite{Bo-Na}.
Essentially, this formula appeared in Dewar-Richter \cite{De-Ri}. It is not explicitly proved that $F|A(p)$ is a true modular form, we prove it here for more general operators $A^{(j)}(M)$.

Let $M$ be a positive integer. For a formal Fourier series of the form $F=\sum _{T\in \Lambda _n}a_F(T)q^T$, we set
\begin{align*}
F|A^{(j)}(M):= \sum _{T^{[j]}\equiv 0 \bmod{M}} a_F(T)q^T.
\end{align*}
Then we can prove its modularity as follows:
\begin{Thm}
\label{AOP2}
Let $k$, $j$, $n$ ($j\le n$), $M$ and $N$ be positive integers.
If $F\in M_k(\Gamma _1^{(n)}(N))$ then $F|A^{(j)}(N)\in M_k(\Gamma _1^{(n)}(NM^2))$.
In particular, if $F\in M_k(\Gamma _0^{(n)}(N), \chi)$ for a Dirichlet character $\chi $ modulo $N$, then $F|A^{(j)}(M)\in M_k(\Gamma _0^{(n)}(NM^2),\chi )$.
\end{Thm}
\begin{Rem}
As a special case in the above, we have
\begin{align*}
&F|A^{(n)}(M)=F|A(M)=\sum _{\substack{T\in \Lambda _n \\ M\mid \det T}}a_F(T)q^T,\\
&F|A^{(1)}(M)=F|U(M)V(M)=\sum _{T\in \Lambda _n}a_F(MT)q^{MT}.
\end{align*}
Here $U(M)$ and $V(M)$ are the usual operator described as
\begin{align*}
F|U(M)=\sum _{T\in \Lambda _n}a_F(MT)q^T,\quad  F|V(M)=\sum _{T\in \Lambda _n}a_F(T)q^{MT}.
\end{align*}
For more details, see \cite{Bo2}.
\end{Rem}
\begin{proof}
We put $J:=\{T\bmod{N}\;|\;T\in \Lambda _n \}$. Note that $J$ is a finite set. Then we can find as in \cite{Bo-Na3} that
\[\sum _{T\equiv T_0 \bmod{M}} a_F(T)q^T\in M_k(\Gamma _1^{(n)}(NM^2))\]
for any $T_0\in J$. Now we consider
\[J_0^{(j)}:=\{T \bmod{M}\;|\;T\in \Lambda_n,\ T^{[j]}\equiv 0 \bmod{M} \}\subset J.\]
Then we have
\begin{align*}
F|A^{(j)}(M)=\sum _{T_0\in J_0^{(j)}}\sum _{T\equiv T_0 \bmod{M}}a_F(T)e^{2\pi i {\rm tr }(TZ)}.
\end{align*}
Hence $F|A^{(j)}(M)\in M_k(\Gamma _1^{(n)}(NM^2))$.

Assume that $F\in M_k(\Gamma _0^{(n)}(N),\chi )$.
Using the standard procedure of twisting, we show
that $F|A^{(j)}(M)\in M_k(\Gamma_0^{(n)}(NM^2),\chi )$.
If $g=\left(\begin{smallmatrix}
A & B\\ C & D\end{smallmatrix}\right)\in \Gamma^{(n)}_0(NM^2)$ we get,
using \cite{Bo-Na3}
\begin{equation}
F| A^{(j)}(M)|_k\: g=
\sum_{T_0\in J_0^{(j)}} \sum_S F |_k \left(\begin{array}{cc} 1_n & \frac{S}{M}\\
0_n & 1_n\end{array}\right)e^{-2\pi i tr(T_0 \frac{S}{M})}|_k\: g,
\label{twist1}
\end{equation}
where $S$ runs over all symmetric integral matrices of size $n$ modulo $M$.
Then we easily get
$$  \left(\begin{array}{cc} 1_n & \frac{S}{M}\\
0_n & 1_n\end{array}\right)\cdot g= \tilde{g}\cdot\left(\begin{array}{cc}
1_n & \frac{\tilde{S}}{M}\\
0_n & 1_n\end{array}\right)$$
with $\tilde{g}\in \Gamma_0^{(n)}(N)$ and some integral symmetric $\tilde{S}$
satisfying $\tilde{S}\equiv {}^t\!DSD  \bmod M$.
Note here that $\tilde{g}$ satisfies $\tilde{g}\equiv \left( \begin{smallmatrix} * & * \\ * & D \end{smallmatrix} \right)$ mod $M$.
Keeping in mind that $A{}^t\!D\equiv 1_n\bmod M$ we may rewrite (\ref{twist1}) as
$$\sum_{T_0} \sum_{\tilde{S}} \left(F|_k\: \tilde{g}\right)|_k
\left(\begin{array}{cc} 1_n & \frac{\tilde{S}}{M}\\ 0_n & 1_n\end{array}\right)
e^{-2\pi i tr(-{}^t\!AT_0A\frac{\tilde{S}}{M})}.
$$
We observe that  $F|_k\: \tilde{g}=\chi (\det D)F$ and
$T_0 \longmapsto \tilde{T_0}:= -{}^t\!A T_0A$ just permutes the
set $J_0^{(j)}$; this proves the assertion.
\end{proof}
\begin{Rem}
  %\begin{enumerate}
    %\item
The proof actually shows that $F| A^{(j)}(M)$ is a modular form of
    level ${\rm lcm} (M^2,N)$.\\
    %\item
(2) We have the same statement as in Theorem \ref{AOP2} for $F$ of half integral weight in the following way:
We consider $$G(Z):= F\cdot \theta^{(n)}(MZ),$$
where $\theta^{(n)}(Z)$ is the theta function introduced in Subsection \ref{sec:siegel-modular-forms}.
This is of integral weight and (obviously) we have
$$F|A^{(j)}(M) =(F|A^{(j)}(M)) \cdot \theta ^{(n)}(MZ). $$
Therefore, the statement for $F$ follows from that for $G$.

%By using the action of \begin{math} \left( \left( \begin{smallmatrix}  1_{n} & S/M\\  0_{n} & 1_{n}   \end{smallmatrix}\right) ,  1 \right) \in \mathfrak{G},   \end{math} the same proof works for half-integral weight case. Here $\mathfrak{G}$ is the group defined in Subsection \ref{sec:siegel-modular-forms}.
  %\end{enumerate}
\end{Rem}

For any $F\in M_k(\Gamma _n)$, we have $F|A^{(j)}(p^m)\in M_k(\Gamma _0^{(n)}(p^{2m}))$ by Theorem \ref{AOP2}. By Theorem \ref{p-adic}, we can regard as $A^{(j)}(p^m): \widetilde{M}(\Gamma _n)_{p^l}\longrightarrow \widetilde{M}(\Gamma _n)_{p^l}$.
\begin{Prop}
For any $l\ge 1$, $m\ge 1$ and $j$ ($1\le j \le n$), we can decompose $\widetilde{M}(\Gamma _n)_{p^l}$ as
$$\widetilde{M}(\Gamma _n)_{p^l}={\rm Ker}A^{(j)}(p^m)\bigoplus {\rm Im}A^{(j)}(p^m).$$
\end{Prop}

\begin{proof}
Let $\widetilde{F}\in \widetilde{M}(\Gamma _n)_{p^l}$. We set
\begin{align*}
\widetilde{F_1}:=\sum_{T^{[j]} \not \equiv 0 \bmod{p^m}}\widetilde{a_F(T)}q^T,\quad \widetilde{F_2}:=\sum _{T^{[j]}\equiv 0 \bmod{p^m}}\widetilde{a_F(T)}q^{T}.
\end{align*}
Namely $\widetilde{F_1}:=\widetilde{F}-\widetilde{F}|A^{(j)}(p^m)$ and $\widetilde{F_2}:=\widetilde{F}|A^{(j)}(p^m)$. Then $\widetilde{F}$ can be written as $\widetilde{F}=\widetilde{F_1}+\widetilde{F_2}$. Then $\widetilde{F_1}\in {\rm Ker}A^{(j)}(p^m)$, $\widetilde{F_2}\in {\rm Im}A^{(j)}(p^m)$. This shows $\widetilde{M}(\Gamma _n)_{p^l}\subset{\rm Ker}A^{(j)}(p^m)+ {\rm Im}A^{(j)}(p^m)$. The converse inclusion is trivial. Therefore
\[\widetilde{M}(\Gamma _n)_{p^l}={\rm Ker}A^{(j)}(p^m)+ {\rm Im}A^{(j)}(p^m).\]

We shall prove that the summation of the right hand side is direct. Let $\widetilde{F}\in {\rm Ker}A^{(j)}(p^m) \cap {\rm Im}A^{(j)}(p^m)$. Then $\widetilde{F}=\widetilde{G}|A^{(j)}(p^m)$ for some $\widetilde{G}\in \widetilde{M}(\Gamma _n)_{p^l}$. This implies
\[\widetilde{F}=\sum _{T^{[j]}\equiv 0 \bmod{p^m}}\widetilde{a_G(T)}q^T.\]
On the other hand, it follows from $\widetilde{F}\in {\rm Ker}A^{(j)}(p^m)$ that
\[ \widetilde{F}|A^{(j)}(p^m)=\widetilde{F}=0.\]
Hence we have ${\rm Ker}A^{(j)}(p^m)\cap {\rm Im}A^{(j)}(p^m)=0$.
\end{proof}
\begin{Rem}
Similarly we have also
$$\widetilde{M}(\Gamma _n)_{p^l}={\rm Ker}\Theta ^m\bigoplus {\rm Im}\Theta ^{m}, $$
for any $l$ and $m$ with $1\le l\le m$.
\end{Rem}

We consider the action of $A^{(j)}(p^m)$ on the space of $p$-adic modular forms:
\begin{Prop}
If $F$ is a $p$-adic modular form of degree $n$, then $F|A^{(j)}(p^m)$ is a $p$-adic modular form of degree $n$ for any $1\le j\le n$ and $m\ge 1$.
\end{Prop}
\begin{proof}
Since $F$ is a $p$-adic modular form, there exists a sequence $\{G_{l}\in M_{k_l}(\Gamma _n)_{\mathbb{Z}_{(p)}}\}_{l}$ such that $F\equiv G_l$ mod $p^l$. Then we have $F|A^{(j)}(p^m) \equiv G_{l}|A^{(j)}(p^m)$ mod $p^l$. By Theorem \ref{AOP2}, $G_l|A^{(j)}(p^m)\in M_{k_l}(\Gamma _0^{(n)}(p^{2m}))_{\mathbb{Z}_{(p)}}$ holds. By Theorem \ref{p-adic}, $G_l|A^{(j)}(p^m)$ is also a $p$-adic modular form. Therefore $F|A^{(j)}(p^m)$ is a limit of a sequence of $p$-adic modular forms. This implies the assertion.
\end{proof}

%%%%%%%%%%%%%%%%%%%%%%%%

%%%%%%%%%%%%%%%%%%%%%%%%%5
\section{Examples}
\label{Ex:0}
In this section, we introduce some examples of elements of the mod $p$ kernel of $\Theta ^{[j]}$ and analyze the filtrations of some of them.
%%%%%%%%%%%%%%%%%%%%%%%%%%%%%%%%%%%%%%%%%%%%%
\subsection{By the Siegel-Eisenstein series}
Let $E_k^{(n)}$ be the Siegel-Eisenstein series of weight $k$ of degree $n$, where $k > n + 1$ is an even integer.
Let $p$ be a prime and $n$ a positive even integer such that $p\equiv (-1)^{\frac{n}{2}}$ mod $4$ and $p>n+3$.
We set $k_{(n,p)}:=\frac{n+p-1}{2}$. By Nagaoka's result \cite{Na}, $E^{(n)}_{k_{(n,p)}}$ is an element of the mod $p$ kernel of $\Theta$. Note that this is mod $p$ non-singular by the result of \cite{Bo-Ki}.
It follows from $k_{(n,p)}<p$ that
\[\omega _1 \left(E^{(n)}_{k_{(n,p)}}\right)=k_{(n,p)}.\]

As an easy application of Theorem \ref{Main1} (1), we can prove $\Theta ^{[n-1]}\left(E^{(n)}_{k_{(n,p)}}\right)\not \equiv 0$ mod $p$ as follows: We can find an integer $1\le j_0\le n-1$ such that $j_0$ is the max of positive integers $j$ satisfying $\Theta ^{[j]}\left(E^{(n)}_{k_{(n,p)}}\right)\not \equiv 0$ mod $p$. Applying Theorem \ref{Main1} (1), we have $n+p-1-j_0=p$.
This implies $j_0=n-1$.

%%%%%%%%%%%%%%%%%%%%%%%%%%%%%%%%%%%%%%%%%%%%%%%
\subsection{By theta series}
In \cite{Bo-Ko-Na}, we use certain theta series attached to quadratic forms
to construct several types of modular forms
in the kernel of theta operators mod $p$. We compute $\omega_N$ for some
cases, here $N$ can be an arbitrary number coprime to $p$.
In this way we confirm that the constructions in \cite{Bo-Ko-Na} are the best
possible ones in the sense that the level one forms obtained are
of smallest possible weight.

%%%%%%%%%%%%%%%%%%%%%%%%%%%
\paragraph{First case:} Here $S$ is an even positive definite quadratic form
of (even) rank $n$, exact level $p$ and $\det(S)=p^2$.

We showed that the normalized theta series
$$\theta_S^{(n)}(Z):=\frac{1}{\sharp {\rm Aut}_{\mathbb Z}(S)}\sum_{X\in {\mathbb Z}^{(n,n)}} e^{\pi i {\rm tr}({}^t\! XSXZ)}$$
is congruent mod $p$ to a level one form $F$ of weight
$$k=\frac{n}{2}+(p-1),$$
where
$${\rm Aut}_{\mathbb Z}(S)=\{A\in {\mathbb Z}^{(n,n)}\, \mid\, {}^t\!ASA=S\}.$$
Then $a_F(S)=1$,
in particular, $F\not \equiv 0$ mod $p$ and
$$\Theta^{[n-1]}(F)\equiv 0\bmod{p}, \qquad \Theta^{[n-2]}(F)\not \equiv 0\bmod p.$$
Here the second statement follows from
$$S^{[n-2]}\cdot a_F(S)\not \equiv 0\bmod p.$$
Since $j=n-2$ (even) and
$$k-\frac{j}{2}=k-\frac{n-2}{2}= p<p^2-p, $$
we can apply Theorem \ref{Main2} (2) to $F$. This implies
$$k=\omega _1(F). $$
%%%%%%%%%%%%%%%%%%%%%%
\paragraph{Second case:}
Here $S$ is an even positive quadratic form of (even) rank $n$, exact level $p$
with $\det(S)=p$. In this case, $\theta^{(n)}_S$ is congruent mod $p$ to a level one
form of weight $$k=\frac{n}{2}+\frac{p-1}{2}.$$
Then $\Theta^{[n]}(F)\equiv 0$ mod $p$ but $\Theta^{[n-1]}(F)\not \equiv 0$ mod $p$.
Since $j=n-1$ (odd) and $2k-j=2k-n+1=p<p^2-p+1$, we can apply Theorem \ref{Main2} (2). Therefore, in this case also we have
$$k=\omega_1(F).$$

%%%%%%%%%%%%%%
\paragraph{Second case, with harmonic polynomial:}
Let $S$ be as
before and consider
$$\theta^{(n)}_{S,\det}(Z):=\sum_{X\in {\mathbb Z}^{(n,n)}} \det(X)e^{\pi i {\rm tr}({}^t\!XSXZ)}$$
Here we must assume in addition that ${\rm Aut}_{\mathbb Z}(S)$ does not contain
improper automorphisms (i.e. all automorphisms have determinant $+1$).
Then it was shown in \cite{Bo-Ko-Na} that this theta series is congruent to
a (cuspidal) level one modular form $F$ of weight
$$k=\frac{n}{2}+1+3\frac{p-1}{2}.$$
The proof was much more complicated than in the other cases.
Again $F$ satisfies
$\Theta^{[n]}(F)\equiv 0$ mod $p$, but $\Theta^{[n-1]}(F)
\not \equiv 0$ mod $p$.

In this case, from $j=n-1$ (odd) we have
\begin{align*}
2k-j=n+2+3(p-1)-(n-1)=3p.
\end{align*}
Then $p\mid (2k-j)$ and $2k-j=3p<p^2-p+1$ (when $p\ge 5$). Applying Theorem \ref{Main1} (2), we have
$$k=\omega_1 (F).$$

\begin{Rem} There is a missing case here, namely
$\theta^{(n)}_{S,\det}$ with $S$ of level $p$ and $\det(S)=p^2$.
Here we do not know yet a good explicit construction of a level one form $F$
congruent to $\theta^{(n)}_{S,\det}$ with low weight. A consideration similar
to the one above suggests that
$$\omega_1(F)= \frac{n}{2}+1+2(p-1)$$
should hold.
\end{Rem}
%%%%%%%%%%%%%%%%%%%%%%%%%%%%%%%%%%%%%%%%%%
\subsection{By the operators $A^{(j)}(p)$}
\label{Ex:2}
For any $F\in M_k(\Gamma _n)_{\mathbb{Z}_{(p)}}$, we have $\Theta ^{[j]}(F|A^{(j)}(p))\equiv 0$ mod $p$ by the definition of $A^{(j)}(p)$. Namely, we can always construct elements of the mod $p$ kernel of $\Theta^{[j]}$ for any prime $p$. In particular, if $p\ge n+3$ and $j=n$, then we get of weight $k+p^2-1$ for any $k\in \mathbb{Z}_{\ge 1}$ (see Section \ref{Sec:AOP}). Moreover these examples are not necessarily of type (b) introduced in Subsection \ref{Main}. Because, we take a suitable $F$, then there exists $T\in \Lambda _n$ with $p\nmid \varepsilon (T)$ such that $a_{F|A^{(j)}(p)}(T)\not \equiv 0$ mod $p$. We remark that $F$ is an element of the mod $p$ kernel of $\Theta ^{[j]}$ if and only if $F|A^{(j)}(p)\equiv F$ mod $p$.

Let $X_{10}^{(2)}, \ X_{12}^{(2)}$ be cusp forms of degree $2$, level $1$
and weight $10, \ 12$ respectively.
We normalize them so that
\begin{math}
  a_{X_{10}}\left(\smat{1}{1/2}{1/2}{1}\right) =
  a_{X_{12}}\left(\smat{1}{1/2}{1/2}{1}\right) = 1.
\end{math}

\begin{Ex}
 We have by direct calculation
\[\omega _1(E^{(2)}_{8}|A(7))=32,\ \omega _1(E^{(2)}_{10}|A(7))=4,\ \omega _1(X_{10}|A(7))= 46.\]
All of these formulas satisfy $p\mid (2\omega _1(F)-n+1)$. Therefore Conjecture \ref{Conj} is true for these examples.
\end{Ex}

We introduce more examples of $\omega_{1}(F|A(p))$ in tables of Section \ref{Table}.
We shall explain the tables: Let $p \ge 5$ be a prime number and
$R_{p} = \FF_{p}[x_{4}, x_{6}, x_{10}, x_{12}]$ a polynomial ring over $\FF_{p}$.
For a positive integer $k$, we denote by $R_{p, k} \subset R_{p}$ the space of
isobaric polynomials of weight $k$.
We define a linear map $\psi_{k}: R_{p, k} \rightarrow \tM_{k}(\Gamma_2)_p$
by $\psi_{k}(f(x_{4}, x_{6}, x_{10}, x_{12})) = f(\tE_{4}, \tE_{6}, \tX_{10}, \tX_{12})$.
Then by Nagaoka \cite{nagaoka2000note}, $\psi_{k}$ is an isomorphism.
Let $H \in \tM_{p-1}(\Gamma_2)_{p}$ be a modular form with $H = 1$.
Therefore
\begin{equation*}
  H =
  \begin{cases}
    \tE_{4}& \text{ if } p = 5,\\
    \tE_{6} & \text{ if } p = 7.
  \end{cases}
\end{equation*}
Note that $\psi_{k}^{-1}(H)$ is a prime element of $R_{p}$.
For $F \in \tM_{k}(\Gamma_2)_{p}$, denote by $\mathrm{ord}_{H}(F)$
the maximum integer $e$ such that $\psi_{k}^{-1}(F)/(\psi_{k}^{-1}(H))^{e} \in R_{p}$.
We understand $\mathrm{ord}_{H}(0) = \infty$.

We compute images of the linear operator
\begin{equation*}
  A(p): \tM_{k}(\Gamma_2)_{p} \rightarrow \tM_{k + p^{2} - 1}(\Gamma_{2})_{p}
\end{equation*}
for a basis of $\tM_{k}(\Gamma_2)_{p}$ for even $k \le 60$ and $p = 5, 7$.
We fix a basis $\mathcal{B}_{k, p} = \left\{F_{1}, \dots, F_{m}\right\}$ of
$\tM_{k}(\Gamma_2)_{p}$ so that
\begin{equation}
  \label{eq:bkp-cond}
  \ord_{H}\left(\sum_{G \in S}a_{G}G\right)
  =\min\left\{ \ord_{H}\left(a_{G}G\right)| G \in S \right\},
\end{equation}
for any choice of $a_{G} \in \F_{p}$ for each $G \in S$.
Here $S = \left\{F | A(p): F \in \mathcal{B}_{k, p}\right\}$.
We can take such a basis as follows.
In general, let $R$ be a polynomial ring over a field and $h$ a non-zero element of $R$.
We fix a monomial order of $R$.
Let $\mathcal{M}$ be a subspace of $R$ over $K$
spanned by monomials which are not divisible by the initial term of $h$.
Since $\left\{h\right\}$ is a Gr\"{o}bner basis of the ideal $Rh$,
we can perform the reduction algorithm uniquely.
That is, for any $f \in R$, we can uniquely write $f$ as
\begin{equation}
  \label{eq:red-exp}
  f = \sum_{i = 0}^{\infty}g_{i}(f) h^{i},
\end{equation}
where $g_{i}(f) \in \mathcal{M}$ and $g_{i}(f) = 0$ for sufficiently large $i$.
Take a basis $\mathcal{B} = \left\{F_{1}, \dots, F_{m}\right\}$ of
$\tM_{k}(\Gamma_{2})_{p}$.
We put $f_{i} = \psi_{k + p^{2}-1}^{-1}\left(F_{i}|A(p)\right)$ and denote $g_{j}(f_{i})$
by the element of $\mathcal{M}$ as in \eqref{eq:red-exp} for $h = \psi^{-1}_{p - 1}(H)$.
Take a positive integer $a$ so that $g_{j}(f_{i})=0$ for all $j > a$.
Let $\mathcal{M}'$ be a subspace of $\mathcal{M}$ spanned by
$\left\{g_{j}(f_{i})\right\}_{1\le i \le m, 0 \le j}$. We fix a linear isomorphism
$\Psi: \mathcal{M}' \cong \F_{p}^{\nu}$ and put $v(f_{i}) =
\Psi(g_{0}(f_{i})) \oplus \Psi(g_{1}(f_{i}))\oplus \dots \oplus\Psi(g_{a}(f_{i}))\in \F_{p}^{\nu(a + 1)}$.
If we take a basis $\mathcal{B}$ so that the matrix $\left(v(f_{1}), \dots, v(f_{m})\right)$
is an echelon form, then the basis satisfies the condition \eqref{eq:bkp-cond}.

For example, we take $\mathcal{B}_{18, 5}$ as
\begin{equation*}
  \mathcal{B}_{18, 5} =
  \left\{
    \tEE_{4}^{3}\tEE_{6},\
    \tEE_{4}^{3}\tEE_{6} + 2\tEE_{6}^{3},\
    \tEE_{6}\tX_{12},\ \tEE_{4}^{2}\tX_{10}
  \right\}.
\end{equation*}
Here we simply write $E_{k}^{(2)}$ as $E_{k}$.
Then its images are given as
\begin{align*}
  \tEE_{4}^{3}\tEE_{6} | A(5) &= \tEE_{4}^{6} \tEE_{6}(\tEE_{4}^{3} - 2\tX_{12})\\
  \left(\tEE_{4}^{3}\tEE_{6} + 2\tEE_{6}^{3}\right)|A(5)
                              &= \tEE_{4}^{6}(\tEE_{4}^{3}\tEE^{6} + 2\tEE_{6}^{3} +
                                \tEE_{4}^{2}\tX_{10}),\\
  \tEE_{6}\tX_{12} | A(5) &= \tEE_{4}^{2}\tX_{10} | A(5) = 0.
\end{align*}
We omit the explicit description of $\mathcal{B}_{k, p}$ for other cases.
We note that the multiset
\begin{equation*}
  \left\{\mathrm{ord}_{H}(F|A(p))\bigm | F \in \mathcal{B}_{k, p}\right\}
\end{equation*}
does not depend on the choice of $\mathcal{B}_{k, p}$ satisfying \eqref{eq:bkp-cond}.
Define a map $\alpha_{k, p}: \mathcal{B}_{k, p} \rightarrow \Z^{2}\times \FF_{p}^{2}$
by
\begin{equation*}
  \alpha_{k, p}(F) = \left(
    \mathrm{ord}_{H}(F|A(p)),
    l, l \bmod{p}, 2l - 1 \bmod{p}
  \right),
\end{equation*}
where $l = \omega_{1}(F|A(p))$.
Tables \ref{tab:a5-image} and \ref{tab:a7-image} show the multiset
\begin{equation*}
  \left\{\alpha_{k, p}(F) \bigm | F \in \mathcal{B}_{k, p}\right\}.
\end{equation*}
Each element $[(a, b, c, d), e]$ in tables means that there exists exactly
$e$ modular forms $F \in \mathcal{B}_{18, 5}$ such that $\alpha_{k, p}(F) = (a, b, c, d)$.

Examples show that
there exists a modular form $F$ with $\omega_{1}(F|A(p)) \not \equiv 0$ mod $p$
and $2\omega_{1}(F|A(p)) - 1 \not \equiv 0$ mod $p$.
Filtrations of such modular forms in these tables are
$\left\{24, \ 42, \ 54\right\}$ if $p = 5$ and
$\left\{24, \ 48, \ 52\right\}$ if $p = 7$.
For example, a modular form of degree $2$, weight $24$, level $1$
\begin{equation*}
  F = E_{4}^{3} E_{6}^{2} + 2 E_{6}^{4} + 3 E_{4}^{2} E_{6} X_{10} +
  3 E_{4} X_{10}^{2} + 2 E_{6}^{2} X_{12} + 3 X_{12}^{2}
\end{equation*}
satisfies $\Theta^{[2]}(F) \equiv 0$ mod $5$ and $\Theta^{[1]}(F) \not \equiv 0$ mod $5$,
but we have
\begin{equation*}
  2\omega_{1}(F) - 1 \not \equiv 0 \mod{5}.
\end{equation*}
Bold elements in tables indicate those modular forms.
%%%%%%%%%%%%%%%%%%%%%%%%%
\subsection{More examples of filtrations}
\label{sec:example-filtr-case}
In this subsection, we show example of $\omega_{1}(F)$ for $F \in M_{k}(\Gamma_{2})_{\Z_{(p)}}$ with
$\Theta(F) \equiv 0$ mod $p$ to test the validity of Conjecture \ref{Conj}.
Here we compute the kernel of
$\Theta : \tM_{k}(\Gamma_{2})_{p} \rightarrow \tM_{k + p + 1}(\Gamma_{2})_{p}$
for $p < 80$ and an even $k \le 100$ with $b_{k + p + 1} \le 15$.
Here
\begin{equation*}
  b_{k} =
  \begin{cases}
    [k/10] & \text{ if } k \text{ is even},\\
    [(k - 5)/10] & \text{ if } k \text{ is odd}.
  \end{cases}
\end{equation*}
Note that $b_{k}$ gives the Sturm bound for $\tM_{k}(\Gamma_{2})_{p}$ (cf. \cite{C-C-K,Ki-Ta}).
We take a basis $\mathcal{B} = \left\{F_{1}, \dots, F_{m}\right\}$
of $\mathrm{Ker}\Theta$ so that the condition \eqref{eq:bkp-cond} holds for $S = \mathcal{B}$.

We understand $\ord_{H}(F)$ for $F \in \tM_{k}(\Gamma_{2})_{p}$ with an odd $k$ as follows:
By Nagaoka \cite{nagaoka2005note}, there uniquely exists $G \in \tM_{k - 35}(\Gamma_{2})_{p}$
such that $F = X_{35}G$, where $X_{35}$ is the Igusa's generator of weight $35$. Then we define
$\ord_{H}(F)= \ord_{H}(G)$.

Then we have computed the filtration $\omega_{1}(F_{i})$ for $i = 1, \dots, m$.
The table \ref{tab:wt-kernel-theta2} is of these filtrations.
The meaning of the table is as follows:
For a prime $p$, a positive integer $k$ appears in the corresponding cell if and only if
$k\le 100$, $b_{k + p + 1} \le 15$ and there exists $F \in \tM_{k}(\Gamma_{2})_{p}$
such that $F \ne 0$, $\omega_{1}(F) = k$ and $\Theta(F)=0$.

Table \ref{tab:wt-kernel-theta2} also shows that
there exists $F \in \tM_{k}(\Gamma_{2})_{p}$ such that
$\Theta(F) = 0$, $\omega_{1}(F) \not \equiv 0 \mod{p}$ and
$2\omega_{1}(F) - 1 \not \equiv 0 \mod{p}$.
The pairs $(p, \omega_{1}(F))$ for such $F$ in the table are
\begin{gather*}
  (5, 24), (5, 42), (5, 54), (5, 66), (5, 72), (5, 74), (5, 84), (5, 92), (5, 96), \\
  (7, 24), (7, 48), (7, 52), (7, 72), (7, 76), (7, 80), (7, 94), (7, 96), (11, 60), (13, 84).
\end{gather*}

%%%%%%%%%%%%%%%%%%%
\section{Tables}
\label{Table}
%%%%%%%%%%%%%%%%%%%%%%%%%%%%%%%%%%%%
\subsection{Tables for filtrations of images of $A(p)$}
The following tables are of $[(a,b,c,d),e]$, where
\begin{align*}
(a,b,c,d):=\alpha_{k, p}(F) = \left(
    \mathrm{ord}_{H}(F|A(p)),
    l, l \bmod{p}, 2l - 1 \bmod{p}
  \right),
\end{align*} $l = \omega_{1}(F|A(p))$, and $e$ is the number of modular forms which have $\alpha_{k, p}(F)$.
For more details, see Subsection \ref{Ex:2}

\begin{longtable}{c|l}
  \caption{Filtrations of images of $A(5)$}
  \label{tab:a5-image}\\
  4 & \parbox{28pc}{$[(7, 0, 0, 4), 1]$} \\ \hline
  6 & \parbox{28pc}{$[(3, 18, 3, 0), 1]$} \\ \hline
  8 & \parbox{28pc}{$[(8, 0, 0, 4), 1]$} \\ \hline
  10 & \parbox{28pc}{$[(4, 18, 3, 0), 1]$, $[(\infty, 0, 0, 4), 1]$} \\ \hline
  12 & \parbox{28pc}{$[(2, 28, 3, 0), 1]$, $[(9, 0, 0, 4), 1]$, $[(\infty, 0, 0, 4), 1]$} \\ \hline
  14 & \parbox{28pc}{$[(5, 18, 3, 0), 1]$, $[(\infty, 0, 0, 4), 1]$} \\ \hline
  16 & \parbox{28pc}{$[(3, 28, 3, 0), 1]$, $[(10, 0, 0, 4), 1]$, $[(\infty, 0, 0, 4), 2]$} \\ \hline
  18 & \parbox{28pc}{$[(6, 18, 3, 0), 2]$, $[(\infty, 0, 0, 4), 2]$} \\ \hline
  20 & \parbox{28pc}{$[(4, 28, 3, 0), 2]$, $[(11, 0, 0, 4), 1]$, $[(\infty, 0, 0, 4), 2]$} \\ \hline
  22 & \parbox{28pc}{$[(2, 38, 3, 0), 1]$, $[(7, 18, 3, 0), 2]$, $[(\infty, 0, 0, 4), 3]$} \\ \hline
  24 & \parbox{28pc}{$[(0, 48, 3, 0), 2]$, $[(5, 28, 3, 0), 2]$, $\mathbf{[(6, 24, 4, 2), 1]}$, $[(12, 0, 0, 4), 1]$, $[(\infty, 0, 0, 4), 2]$} \\ \hline
  26 & \parbox{28pc}{$[(3, 38, 3, 0), 1]$, $[(8, 18, 3, 0), 2]$, $[(\infty, 0, 0, 4), 4]$} \\ \hline
  28 & \parbox{28pc}{$[(1, 48, 3, 0), 2]$, $[(6, 28, 3, 0), 2]$, $\mathbf{[(7, 24, 4, 2), 1]}$, $[(13, 0, 0, 4), 1]$, $[(\infty, 0, 0, 4), 4]$} \\ \hline
  30 & \parbox{28pc}{$[(4, 38, 3, 0), 1]$, $[(6, 30, 0, 4), 1]$, $[(9, 18, 3, 0), 2]$, $[(\infty, 0, 0, 4), 7]$} \\ \hline
  32 & \parbox{28pc}{$[(2, 48, 3, 0), 3]$, $[(7, 28, 3, 0), 2]$, $\mathbf{[(8, 24, 4, 2), 1]}$, $[(14, 0, 0, 4), 1]$, $[(\infty, 0, 0, 4), 5]$} \\ \hline
  34 & \parbox{28pc}{$[(0, 58, 3, 0), 3]$, $[(5, 38, 3, 0), 1]$, $[(7, 30, 0, 4), 1]$, $[(10, 18, 3, 0), 2]$, $[(\infty, 0, 0, 4), 7]$} \\ \hline
  36 & \parbox{28pc}{$[(3, 48, 3, 0), 4]$, $[(8, 28, 3, 0), 2]$, $\mathbf{[(9, 24, 4, 2), 1]}$, $[(15, 0, 0, 4), 1]$, $[(\infty, 0, 0, 4), 9]$} \\ \hline
  38 & \parbox{28pc}{$[(1, 58, 3, 0), 3]$, $[(6, 38, 3, 0), 1]$, $[(8, 30, 0, 4), 1]$, $[(11, 18, 3, 0), 2]$, $[(\infty, 0, 0, 4), 9]$} \\ \hline
  40 & \parbox{28pc}{$[(4, 48, 3, 0), 5]$, $[(9, 28, 3, 0), 2]$, $\mathbf{[(10, 24, 4, 2), 1]}$, $[(16, 0, 0, 4), 1]$, $[(\infty, 0, 0, 4), 12]$} \\ \hline
  42 & \parbox{28pc}{$[(2, 58, 3, 0), 4]$, $\mathbf{[(6, 42, 2, 3), 1]}$, $[(7, 38, 3, 0), 1]$, $[(9, 30, 0, 4), 1]$, $[(12, 18, 3, 0), 2]$, $[(\infty, 0, 0, 4), 13]$} \\ \hline
  44 & \parbox{28pc}{$[(0, 68, 3, 0), 3]$, $[(5, 48, 3, 0), 5]$, $[(10, 28, 3, 0), 2]$, $\mathbf{[(11, 24, 4, 2), 1]}$, $[(17, 0, 0, 4), 1]$, $[(\infty, 0, 0, 4), 12]$} \\ \hline
  46 & \parbox{28pc}{$[(3, 58, 3, 0), 5]$, $\mathbf{[(7, 42, 2, 3), 1]}$, $[(8, 38, 3, 0), 1]$, $[(10, 30, 0, 4), 1]$, $[(13, 18, 3, 0), 2]$, $[(\infty, 0, 0, 4), 17]$} \\ \hline
  48 & \parbox{28pc}{$[(1, 68, 3, 0), 3]$, $[(6, 48, 3, 0), 7]$, $[(11, 28, 3, 0), 2]$, $\mathbf{[(12, 24, 4, 2), 1]}$, $[(18, 0, 0, 4), 1]$, $[(\infty, 0, 0, 4), 17]$} \\ \hline
  50 & \parbox{28pc}{$[(4, 58, 3, 0), 5]$, $[(6, 50, 0, 4), 1]$, $\mathbf{[(8, 42, 2, 3), 1]}$, $[(9, 38, 3, 0), 1]$, $[(11, 30, 0, 4), 1]$, $[(14, 18, 3, 0), 2]$, $[(\infty, 0, 0, 4), 20]$} \\ \hline
  52 & \parbox{28pc}{$[(2, 68, 3, 0), 4]$, $[(7, 48, 3, 0), 7]$, $[(12, 28, 3, 0), 2]$, $\mathbf{[(13, 24, 4, 2), 1]}$, $[(19, 0, 0, 4), 1]$, $[(\infty, 0, 0, 4), 22]$} \\ \hline
  54 & \parbox{28pc}{$[(0, 78, 3, 0), 6]$, $[(5, 58, 3, 0), 5]$, $\mathbf{[(6, 54, 4, 2), 1]}$, $[(7, 50, 0, 4), 1]$, $\mathbf{[(9, 42, 2, 3), 1]}$, $[(10, 38, 3, 0), 1]$, $[(12, 30, 0, 4), 1]$, $[(15, 18, 3, 0), 2]$, $[(\infty, 0, 0, 4), 21]$} \\ \hline
  56 & \parbox{28pc}{$[(3, 68, 3, 0), 4]$, $[(5, 60, 0, 4), 1]$, $[(8, 48, 3, 0), 7]$, $[(13, 28, 3, 0), 2]$, $\mathbf{[(14, 24, 4, 2), 1]}$, $[(20, 0, 0, 4), 1]$, $[(\infty, 0, 0, 4), 26]$} \\ \hline
  58 & \parbox{28pc}{$[(1, 78, 3, 0), 7]$, $[(6, 58, 3, 0), 5]$, $\mathbf{[(7, 54, 4, 2), 1]}$, $[(8, 50, 0, 4), 1]$, $\mathbf{[(10, 42, 2, 3), 1]}$, $[(11, 38, 3, 0), 1]$, $[(13, 30, 0, 4), 1]$, $[(16, 18, 3, 0), 2]$, $[(\infty, 0, 0, 4), 27]$} \\ \hline
  60 & \parbox{28pc}{$[(4, 68, 3, 0), 4]$, $[(6, 60, 0, 4), 2]$, $[(9, 48, 3, 0), 7]$, $[(14, 28, 3, 0), 2]$, $\mathbf{[(15, 24, 4, 2), 1]}$, $[(21, 0, 0, 4), 1]$, $[(\infty, 0, 0, 4), 35]$}
\end{longtable}

\begin{longtable}{c|l}
  \caption{Filtrations of images of $A(7)$}
  \label{tab:a7-image}\\
    4 & \parbox{28pc}{$[(8, 4, 4, 0), 1]$} \\ \hline
    6 & \parbox{28pc}{$[(9, 0, 0, 6), 1]$} \\ \hline
    8 & \parbox{28pc}{$[(4, 32, 4, 0), 1]$} \\ \hline
    10 & \parbox{28pc}{$[(2, 46, 4, 0), 1]$, $[(9, 4, 4, 0), 1]$} \\ \hline
    12 & \parbox{28pc}{$[(0, 60, 4, 0), 2]$, $[(10, 0, 0, 6), 1]$} \\ \hline
    14 & \parbox{28pc}{$[(5, 32, 4, 0), 1]$, $[(\infty, 0, 0, 6), 1]$} \\ \hline
    16 & \parbox{28pc}{$[(3, 46, 4, 0), 2]$, $[(10, 4, 4, 0), 1]$, $[(\infty, 0, 0, 6), 1]$} \\ \hline
    18 & \parbox{28pc}{$[(1, 60, 4, 0), 2]$, $[(11, 0, 0, 6), 1]$, $[(\infty, 0, 0, 6), 1]$} \\ \hline
    20 & \parbox{28pc}{$[(6, 32, 4, 0), 2]$, $[(\infty, 0, 0, 6), 3]$} \\ \hline
    22 & \parbox{28pc}{$[(4, 46, 4, 0), 2]$, $[(11, 4, 4, 0), 1]$, $[(\infty, 0, 0, 6), 3]$} \\ \hline
    24 & \parbox{28pc}{$[(2, 60, 4, 0), 3]$, $\mathbf{[(8, 24, 3, 5), 1]}$, $[(12, 0, 0, 6), 1]$, $[(\infty, 0, 0, 6), 3]$} \\ \hline
    26 & \parbox{28pc}{$[(0, 74, 4, 0), 2]$, $[(7, 32, 4, 0), 2]$, $[(\infty, 0, 0, 6), 3]$} \\ \hline
    28 & \parbox{28pc}{$[(5, 46, 4, 0), 2]$, $[(8, 28, 0, 6), 1]$, $[(12, 4, 4, 0), 1]$, $[(\infty, 0, 0, 6), 6]$} \\ \hline
    30 & \parbox{28pc}{$[(3, 60, 4, 0), 4]$, $\mathbf{[(9, 24, 3, 5), 1]}$, $[(13, 0, 0, 6), 1]$, $[(\infty, 0, 0, 6), 5]$} \\ \hline
    32 & \parbox{28pc}{$[(1, 74, 4, 0), 3]$, $[(8, 32, 4, 0), 3]$, $[(\infty, 0, 0, 6), 6]$} \\ \hline
    34 & \parbox{28pc}{$[(6, 46, 4, 0), 3]$, $[(9, 28, 0, 6), 1]$, $[(13, 4, 4, 0), 1]$, $[(\infty, 0, 0, 6), 9]$} \\ \hline
    36 & \parbox{28pc}{$[(4, 60, 4, 0), 6]$, $\mathbf{[(10, 24, 3, 5), 1]}$, $[(14, 0, 0, 6), 1]$, $[(\infty, 0, 0, 6), 9]$} \\ \hline
    38 & \parbox{28pc}{$[(2, 74, 4, 0), 4]$, $[(9, 32, 4, 0), 3]$, $[(\infty, 0, 0, 6), 9]$} \\ \hline
    40 & \parbox{28pc}{$[(0, 88, 4, 0), 7]$, $[(7, 46, 4, 0), 3]$, $[(10, 28, 0, 6), 1]$, $[(14, 4, 4, 0), 1]$, $[(\infty, 0, 0, 6), 9]$} \\ \hline
    42 & \parbox{28pc}{$[(5, 60, 4, 0), 7]$, $\mathbf{[(11, 24, 3, 5), 1]}$, $[(15, 0, 0, 6), 1]$, $[(\infty, 0, 0, 6), 13]$} \\ \hline
    44 & \parbox{28pc}{$[(3, 74, 4, 0), 6]$, $[(10, 32, 4, 0), 3]$, $[(\infty, 0, 0, 6), 15]$} \\ \hline
    46 & \parbox{28pc}{$[(1, 88, 4, 0), 8]$, $[(8, 46, 4, 0), 4]$, $[(11, 28, 0, 6), 1]$, $[(15, 4, 4, 0), 1]$, $[(\infty, 0, 0, 6), 13]$} \\ \hline
    48 & \parbox{28pc}{$[(6, 60, 4, 0), 7]$, $\mathbf{[(8, 48, 6, 4), 1]}$, $\mathbf{[(12, 24, 3, 5), 1]}$, $[(16, 0, 0, 6), 1]$, $[(\infty, 0, 0, 6), 21]$} \\ \hline
    50 & \parbox{28pc}{$[(4, 74, 4, 0), 7]$, $[(11, 32, 4, 0), 3]$, $[(\infty, 0, 0, 6), 21]$} \\ \hline
    52 & \parbox{28pc}{$[(2, 88, 4, 0), 9]$, $\mathbf{[(8, 52, 3, 5), 1]}$, $[(9, 46, 4, 0), 4]$, $[(12, 28, 0, 6), 1]$, $[(16, 4, 4, 0), 1]$, $[(\infty, 0, 0, 6), 21]$} \\ \hline
    54 & \parbox{28pc}{$[(0, 102, 4, 0), 8]$, $[(7, 60, 4, 0), 7]$, $\mathbf{[(9, 48, 6, 4), 1]}$, $\mathbf{[(13, 24, 3, 5), 1]}$, $[(17, 0, 0, 6), 1]$, $[(\infty, 0, 0, 6), 21]$} \\ \hline
    56 & \parbox{28pc}{$[(5, 74, 4, 0), 8]$, $[(8, 56, 0, 6), 1]$, $[(12, 32, 4, 0), 3]$, $[(\infty, 0, 0, 6), 30]$} \\ \hline
    58 & \parbox{28pc}{$[(3, 88, 4, 0), 11]$, $\mathbf{[(9, 52, 3, 5), 1]}$, $[(10, 46, 4, 0), 4]$, $[(13, 28, 0, 6), 1]$, $[(17, 4, 4, 0), 1]$, $[(\infty, 0, 0, 6), 28]$} \\ \hline
    60 & \parbox{28pc}{$[(1, 102, 4, 0), 10]$, $[(8, 60, 4, 0), 9]$, $\mathbf{[(10, 48, 6, 4), 1]}$, $\mathbf{[(14, 24, 3, 5), 1]}$, $[(18, 0, 0, 6), 1]$, $[(\infty, 0, 0, 6), 30]$}
\end{longtable}

%%%%%%%%%%%%%%%%%%%%%%%%%%
\subsection{Table for filtrations of the kernel of $\Theta^{[2]}$}
Table \ref{tab:wt-kernel-theta2} shows filtration $\omega_{1}(F)$ for $F \in M_{k}(\Gamma_{2})_{\Z_{(p)}}$
with $\Theta^{[2]}(F) \equiv 0 \mod{p}$, $k \le 100$, $p < 80$ and $b_{k + p + 1} \le 15$.
For more details, see Subsection \ref{sec:example-filtr-case}.
\begin{longtable}{c|l}
  \caption{Filtrations of the kernel of $\Theta^{[2]}$}
  \label{tab:wt-kernel-theta2}\\
  $p$ & $k$\\ \hline
  5 & \parbox{28pc}{$0$, \ $18$, \ $24$, \ $28$, \ $30$, \ $38$, \ $42$, \ $48$, \ $50$,
      \ $54$, \ $58$, \ $60$, \ $66$, \ $68$, \ $72$, \ $74$, \ $78$, \ $80$, \ $83$, \
      $84$, \ $88$, \ $90$, \ $92$, \ $93$, \ $96$, \ $98$}\\ \hline
  7 & \parbox{28pc}{ $0$, \ $4$, \ $24$, \ $28$, \ $32$, \ $46$, \ $48$, \ $52$, \ $56$, \
      $60$, \ $70$, \ $72$, \ $74$, \ $76$, \ $80$, \ $81$, \ $84$, \ $88$, \ $94$, \
      $95$, \ $96$, \ $98$ }\\ \hline
  11 & $0$, \ $6$, \ $28$, \ $44$, \ $50$, \ $60$, \ $66$, \ $72$, \ $83$, \ $88$, \ $94$\\ \hline
  13 & $0$, \ $20$, \ $46$, \ $52$, \ $59$, \ $72$, \ $78$, \ $84$, \ $98$\\ \hline
  17 & $0$, \ $26$, \ $60$, \ $68$, \ $77$, \ $94$\\ \hline
  19 & $0$, \ $10$, \ $48$, \ $76$, \ $86$\\ \hline
  23 & $0$, \ $12$, \ $35$, \ $58$, \ $92$\\ \hline
  29 & $0$, \ $44$\\ \hline
  31 & $0$, \ $16$, \ $47$, \ $78$\\ \hline
  37 & $0$, \ $56$, \ $93$\\ \hline
  41 & $0$, \ $62$\\ \hline
  43 & $0$, \ $22$\\ \hline
  47 & $0$, \ $24$, \ $71$\\ \hline
  53 & $0$, \ $80$\\ \hline
  59 & $0$, \ $30$, \ $89$\\ \hline
  61 & $0$, \ $92$\\ \hline
  67 & $0$, \ $34$\\ \hline
  71 & $0$, \ $36$\\ \hline
  73 & $0$\\ \hline
  79 & $0$, \ $40$
\end{longtable}

%%%%%%%%%%%%%%%%%%%%%%%%%%%%
\section*{Acknowledgment}
The authors would like to thank Professor T.~Yamauchi for informing them
on the filtrations of Siegel modular forms mod $p$.
The second author is supported by JSPS Grant-in-Aid for Young Scientists (B) 26800026.
The  third author is partially supported by JSPS Kakenhi 23224001.
%%%%%%%%%%%%%%%%%%%%%%%%%%%%%%%%%%

% \bibliographystyle{amsplain}
% \bibliography{texref}

\begin{thebibliography}{10}

\bibitem{And-Zu}
A.~N.~Andrianov, V.~G.~Zhuravlev, Modular forms and Hecke operators, AMS Translations
of Mathematical Monographs 145, 1995

\bibitem{Bo-Das}
S.~B\"ocherer, S.~Das, On holomorphic differential operators equivariant for the inclusion of $Sp(n,\bold{R})$ in $U(n,n)$. Int. Math. Res. Not. IMRN  2013,  no. 11, 2534-2567.


\bibitem{Bo} S.~B\"ocherer, \"Uber gewisse Siegelsche Modulformen zweiten Grades.
Math. Ann. 261, 23-41 (1982)

\bibitem{Bo2}
S.~B\"ocherer, On the Hecke operator $U(p)$. With an appendix by Ralf Schmidt. J. Math. Kyoto Univ.  45  (2005),  no. 4, 807-829.

\bibitem{Bo-Ki},
S.~B\"ocherer, T.~Kikuta, On mod $p$ singular modular forms, to appear in Forum Mathematicum.

\bibitem{Bo-Ko-Na}
S.~B\"ocherer, H.~Kodama, S.~Nagaoka, unpublished, preprint

\bibitem{Bo-Na}
S.~B\"ocherer, S.~Nagaoka, On mod $p$ properties of Siegel modular forms,
Math. Ann. 338, 421-433 (2007)

\bibitem{Bo-Na3} S.~B\"ocherer, S.~Nagaoka, Congruences for Siegel modular
forms and their weights, Abh. Math. Semin. Univ. Hambg. 80, 227-231 (2010)

\bibitem{Bo-Na2} S.~B\"ocherer, S.~Nagaoka, On $p$-adic properties of
Siegel modular forms, Automorphic Forms: Research in Number Theory from Oman, Springer Proceedings in Mathematics $\and$ Statistics, 47-66 (2014)

\bibitem{C-C-K}
D.~Choi, Y.~Choie, T.~Kikuta, Sturm type theorem for Siegel modular forms of genus 2 modulo p,
Acta Arith. 158  (2013),  no. 2, 129-139.


\bibitem{C-C-R}
D.~Choi, Y.~Choie, O.~Richter, Congruences for Siegel modular forms,
Annales de l'Institut Fourier, 61 no.4, 1455-1466, (2011)

\bibitem{De-Ri}
Dewar-Richter, Ramanujan congruences for Siegel modular forms. Int. J. Number Theory  6  (2010),  no. 7, 1677-1687.



\bibitem{Ei-Za} M.~Eichler, D.~Zagier, The theory of Jacobi forms,
Progress in Mathematics, vol. 55. Birkh\"{a}user, Boston (1985), v+148 pp.

\bibitem{Frei} E.~Freitag, Siegelsche Modulfunktionen. Grundlehren der math.
Wissenschaften 254 (1983)

\bibitem{Ich}, T.~Ichikawa, Congruences between Siegel modular forms. Math. Ann.  342  (2008),  no. 3, 527-532.

\bibitem{Kat}
N.~M. Katz, A result on modular forms in characteristic $p$.
Modular functions of one variable, V (Proc. Second Internat. Conf., Univ. Bonn, Bonn, 1976),  pp. 53-61. Lecture Notes in Math., Vol. 601, Springer, Berlin, 1977.

\bibitem{Ki-Ko-Na}
T.~Kikuta, H.~Kodama, and S.~Nagaoka, Note on Igusa's cusp form of
  weight $35$, Rocky Mountain Journal of Mathematics \textbf{45} (2015), no. 3.

\bibitem{Ki-Ta}
T.~Kikuta and S.~Takemori,
Sturm bounds for Siegel modular forms of degree 2 and odd weights, arXiv preprint arXiv:1508.01610 (2015).

\bibitem{Kli} H.~ Klingen, Introductory lectures on Siegel modular forms.
Cambridge Univ. Press 1990

% \bibitem{Kob} N.~Koblitz,  $p$-adic variation of the zeta function over families of varieties defined over
% finite fields. Composition Mathematica 31, 119-218 (1975)

\bibitem{Mizu1} S.~ Mizumoto, Fourier coefficients of generalized Eisenstein series of
degree two I.
Invent. math. 65, 115-135 (1981)

\bibitem{Mizu} S.~ Mizumoto,
On integrality of certain algebraic numbers associated with modular forms. Math. Ann. 265 (1983), no. 1, 119-135.

\bibitem{Na} S.~Nagaoka, On the mod $p$ kernel of the theta operator. Proc. Amer. Math. Soc. 143  (2015),  no. 10, 4237-4244.

\bibitem{nagaoka2000note} S.~Nagaoka, Note on mod $p$ Siegel modular forms, Mathe Zeitschrift 235 (2000), no.~2, 405-420.

\bibitem{nagaoka2005note} S.~Nagaoka, Note on mod $p$ Siegel modular forms II, Mathe Zeitschrift 251 (2005), no.~4, 821-826.

\bibitem{Na-Ta}
S.~Nagaoka, S.~Takemori, Notes on theta series for Niemeier lattices. to appear in Ramanujan J.

% \bibitem{Oort}
% F.~Oort, A stratification of a moduli space of Abelian varieties. Progress in Mathematics, Vol 195, (2001)

\bibitem{Se}
J.-P.~Serre, Formes modulaires et fonctions z$\hat{\text{e}}$ta $p$-adiques,
Modular functions of one variable III, Lec. Notes in Math. 350,
Springer Verlag, 1973, 191-268.

\bibitem{Swn}
H.~P.~F.~Swinnerton-Dyer, On $l$-adic representations and congruences
for coefficients of modular forms,
Modular functions of one variable III, Lec. Notes in Math. 350,
Springer Verlag, 1973, 1-55.

\bibitem{Ta}
S.~Takemori, Congruence relations for Siegel modular forms of weight 47, 71, and 89. Exp. Math.  23  (2014),  no. 4, 423-428.

\bibitem{We}
R.~Weissauer, Siegel modular forms mod $p$, arXiv:0804.3134.

\bibitem{We2}
R.~Weissauer,  Vektorwertige Siegelsche Modulformen kleinen Gewichtes. (German) [Vector-valued Siegel modular forms of small weight], J. Reine Angew. Math.  343  (1983), 184-202.
\bibitem{Ya} T.~Yamauchi, The weight reduction of mod $p$ Siegel modular forms for $GSp_4$, arXiv:1410.7894.


\bibitem{Zie}
C.~Ziegler, Jacobi forms of higher degree, Abh. Math. Sem. Univ. Hamburg 59, 191-224 (1989)

\end{thebibliography}

\providecommand{\bysame}{\leavevmode\hbox to3em{\hrulefill}\thinspace}
\providecommand{\MR}{\relax\ifhmode\unskip\space\fi MR }
\providecommand{\MRhref}[2]{%
  \href{http://www.ams.org/mathscinet-getitem?mr=#1}{#2}
}
\providecommand{\href}[2]{#2}

\begin{flushleft}
Siegfried B\"ocherer\\
Mathematisches Institut \\
Universit\"at Mannheim \\
68131 Mannheim, Germany\\
Email: boecherer@t-online.de
\end{flushleft}

\begin{flushleft}
  Toshiyuki Kikuta\\
  Faculty of Information Engineering\\
  Department of Information and Systems Engineering\\
  Fukuoka Institute of Technology\\
  3-30-1 Wajiro-higashi, Higashi-ku, Fukuoka 811-0295, Japan\\
  E-mail: kikuta@fit.ac.jp
\end{flushleft}

\begin{flushleft}
  Sho Takemori\\
  Department of Mathematics, Hokkaido University\\
  Kita 10, Nishi 8, Kita-Ku, Sapporo, 060-0810, Japan \\
  E-mail: takemori@math.sci.hokudai.ac.jp
\end{flushleft}

\end{document}